\documentclass[a4paper]{amsproc}
\usepackage{
 amsmath,
 amsxtra,
 amsthm,
 amssymb,
 etex,
 mathrsfs,
 mathtools,
 tikz-cd,
 xr,
 enumitem,
 float,
 url,
 csvsimple,
 booktabs,
 longtable
 }

\theoremstyle{plain}
 \newtheorem{thm}{Theorem}[section]
 \newtheorem{prop}[thm]{Proposition}
 \newtheorem{lem}[thm]{Lemma}

 \newtheorem{conj}[thm]{Conjecture}
 \newtheorem*{quest*}{Question}
 \newtheorem*{IMC*}{Iwasawa Main Conjecture}
\theoremstyle{definition}
 
 \newtheorem{dfn}[thm]{Definition}
 \newtheorem*{dfn*}{Definition}

 \newtheorem*{hyp*}{Hypothesis}
 \newtheorem*{hyps*}{Hypotheses}
\theoremstyle{remark}
 \newtheorem*{rem*}{Remark}
 \newtheorem*{warn*}{Warning}
 \newtheorem*{rems*}{Remarks}
 \newtheorem*{not*}{Notation}
 \newtheorem*{nots*}{Notations}
 \newtheorem*{convs*}{Conventions}
 \newtheorem*{exm*}{Example}
 \newtheorem*{pf*}{Proof}
 \numberwithin{equation}{section}

\renewcommand{\leq}{\leqslant}
\renewcommand{\geq}{\geqslant}

\newcommand{\QQ}{\mathbb{Q}}

\newcommand{\Cp}{\mathbb{C}_p}
\newcommand{\Zp}{\mathbb{Z}_p}

\newcommand{\cond}{\text{\rm cond}}
\newcommand{\ord}{\text{\rm ord}}
\definecolor{Green}{rgb}{0.0, 0.5, 0.0}
\definecolor{bleu}{rgb}{0.0, 0.0, 0.75}

\title[On Iwasawa $\lambda$-invariants for abelian number fields]{On Iwasawa $\lambda$-invariants for abelian number fields and random matrix heuristics}



\author[Daniel Delbourgo]{\bfseries Daniel Delbourgo}
\address{Department of Mathematics and Statistics \\
University of Waikato \\ Gate 8 \\ Hillcrest Road \\ Hamilton \\ New Zealand\\ 3240}
\email{daniel.delbourgo@waikato.ac.nz}

\author[Heiko Knospe]{\bfseries Heiko Knospe}
\address{Technische Hochschule K\"{o}ln \\ Fakult\"{a}t f\"{u}r Informations-, Medien- und Elektrotechnik \\ Institut f\"{u}r Nachrichtentechnik \\ Betzdorfer Str. 2 \\ 50679 K\"{o}ln \\ Germany}
\email{heiko.knospe@th-koeln.de}

\begin{document}

{\begin{flushleft}\baselineskip9pt\scriptsize\end{flushleft}}
\setcounter{page}{1} \thispagestyle{empty}

\bigskip\bigskip
\begin{abstract}
Following both Ernvall-Mets\"{a}nkyl\"{a} and Ellenberg-Jain-Venkatesh, we study the density of the number of zeroes (i.e. the cyclotomic $\lambda$-invariant)
for the $p$-adic zeta-function twisted by a Dirichlet character $\chi$ of any order. We are interested in two cases: (i) the character $\chi$ is fixed and the prime $p$ varies, and (ii)
$\ord(\chi)$ and the prime $p$ are both fixed but $\chi$ is allowed to vary.
We predict distributions for these $\lambda$-invariants using $p$-adic random matrix theory and provide numerical evidence for these predictions.
We also study the proportion of $\chi$-regular primes, which depends on how $p$ splits inside $\QQ(\chi)$.
Finally in an extensive Appendix, we tabulate the values of the $\lambda$-invariant for every character $\chi$ of conductor $\leq 1000$ and for odd primes $p$ of small size.
\end{abstract}

\maketitle

\tableofcontents

\newpage
\section{Introduction and three conjectures}\label{Sect1}
\noindent
Let $p$ be an odd prime, and suppose that $\chi:\mathbb{Z}\rightarrow\mathbb{C}$ denotes a Dirichlet character.
The Kubota-Leopoldt $L$-function $\mathbf{L}_p(\chi,s)$ is a meromorphic function \cite{KuLe} of the $p$-adic variable $s$, which interpolates
values of the $\chi$-twisted Riemann zeta-function along the critical strip of non-positive integers. More precisely, one has the formula
\begin{equation}\label{InterpFormula}
\mathbf{L}_p(1-n, \chi) \;=\; - \big( 1 -\chi \omega^{-n} (p) p^{n-1} \big) \times \frac{B_{n, \chi \omega^{-n} }}{n}
\quad\text{at every $n\in\mathbb{N}$;}
\end{equation}
here $\omega$ indicates the Teichm\"{u}ller character mod $p$, and $B_{n, \chi \omega^{-n}}$
equals the $n$-th $\chi \omega^{-n}$-twisted Bernoulli number. If $\chi$ is non-trivial then
the $p$-adic zeta-function is analytic in $s$, whilst if $\chi=\mathbf{1}$ then it has a simple pole at $s=1$ with residue $1-\frac{1}{p}$.

Now suppose that $\chi$ is a non-trivial character of the first kind, namely that $\chi = \theta \omega^i$
where $p \nmid \cond(\theta)$ and $i \in \{0,\dots, p-2\}$. We will also set $\mathcal{O}_{\chi}:=\Zp[\text{\rm Im}(\chi)]$.
Iwasawa associated a unique power series $\mathcal{G}_{\chi}(T) = \sum_{j=0}^{\infty} c_j T^j \in \mathcal{O}_{\chi}[\![T]\!]$ satisfying
$$
\mathcal{G}_{\chi}\big( (1+ p)^{-s} - 1\big) =\;
\mathbf{L}_p(s, \chi)
\quad\text{for all $s\in\widehat{\overline{\mathbb{Q}}}_p$ with $\big|s\big|_p\! <p^{\frac{p-2}{p-1}}$.}
$$
The $\mu$-invariant of $\mathcal{G}_{\chi}$ is known to vanish by the Ferrero-Washington theorem \cite{FeWa} in which case this function factorises into a product of an invertible power series,
and a uniquely determined {\em distinguished polynomial}
$$
T^{\lambda} + a_{\lambda -1} T^{\lambda -1 } + \dots + a_1 T + a_0 \in \mathcal{O}_{\chi}[T]
\quad\text{with $\;\big|a_{\lambda-1}\big|_p, \;\dots\;, \big|a_0\big|_p <\; 1$.}
$$
Its degree $\lambda=\lambda_p(\chi)$ equals the number of zeroes of $\mathcal{G}_{\chi}(T)$ on the open unit disk.

\begin{conj}\label{conj1} For even characters $\chi$ of the first kind such that $p \nmid \ord(\chi)$ and $\chi \omega^{-1} (p) \neq 1$, the probability that $\lambda_p(\chi)=r$
is approximately equal to
\begin{equation}
p^{-fr}\times\;\prod_{t>r}\;\big(1-p^{-ft}\big)  \label{pred-prob}
\end{equation}
where $f\geq 1$ denotes the multiplicative order of the prime $p$ modulo $\ord(\chi)$.
\end{conj}

This probability is computed either (i) over primes $p$ and Teichm\"uller twists of a fixed character $\chi$, or (ii) over characters $\chi$ of fixed order $\geq 2$ at a fixed prime $p$.
The condition $p \nmid \ord(\chi)$ ensures that $\mathcal{O}_{\chi}$ is a finite and unramified extension of $\Zp$.
Moreover we have $\mathcal{O}_{\chi} =\Zp\big[\text{\rm Im}(\chi)\big]= \Zp[\zeta_{\ord(\chi)}]$, hence the residue class degree of $\mathcal{O}_{\chi}$ over $\Zp$
will be the smallest positive integer $f$ such that $p^f \equiv 1 \;\big(\!\!\!\!\mod \ord(\chi)\big)$.

\begin{warn*}
If $\chi\omega^{-1}(p)=1$ then $\mathbf{L}_p(s,\chi)$  has a trivial zero at $s=0$, in which case $\mathcal{G}_{\chi}(T)$ has a trivial zero at $T=0$. This happens
precisely  if $\chi = \theta \omega$ where $\theta$ is odd of conductor prime to $p$ with $\theta(p)=1$.
In this situation it is automatically true that $\lambda_p(\chi) \geq 1$,
so the {\it corrected $\lambda$-invariant} is given by $\lambda_p^{\text{corr}}(\chi) = \lambda_p(\chi) -1$.
\end{warn*}

Conjecture \ref{conj1} generalizes a prediction of Ellenberg, Jain and Venkatesh in \cite{EJV}. In their work they also provided numerical evidence for $p \in \{3, 5\} $ and quadratic characters
$\chi$ where $p$ does not split, i.e. at which there is no trivial zero (in the split case they subtracted $1$ from the $\lambda$-invariant). Furthermore, the
first author and Chao \cite{DeQi} studied the $\lambda$-invariant for cubic characters $\chi$ with the prime $p \in \{5, 7\}$,
and found distributions for $\lambda_p(\chi)$ consistent with the quantities predicted in (\ref{pred-prob}).

As we discuss later in Section \ref{Sect2}, following Ellenberg et al. \cite{ChHu,EJV} we employ a $p$-adic random matrix model for the groups $\text{\rm GL}(n,\mathcal{O}_{\chi})$ to formulate
this conjecture. In Section \ref{Sect3} we provide extensive numerical evidence supporting this conjecture at lots of primes and characters: for any given prime $p$ and fixed order, we test all
the characters $\chi$ of conductor $<10000$ and their relevant  Teichm\"{u}ller twists $\chi\omega^i$.
The probability $\lambda_p(\chi)>0$ decreases sharply when the residue class degree $f$ of $\mathcal{O}_{\chi}$ increases. For example, if $p\geq 11$ and $f\geq 2$ then $\text{\rm Prob}(\lambda_p(\chi)>0) \leq 0.01$, and if
$p \geq 101$ and $f \geq 3$ then $\text{\rm Prob}(\lambda_p(\chi) >0) \leq 10^{-6}$.
It follows that the interesting cases are at small primes $p$ and for small orders of $\chi$, which is depicted in Figure \ref{lambdaprob} below.

\vspace{-10mm}

\begin{figure}[H]
  \includegraphics[width=\linewidth]{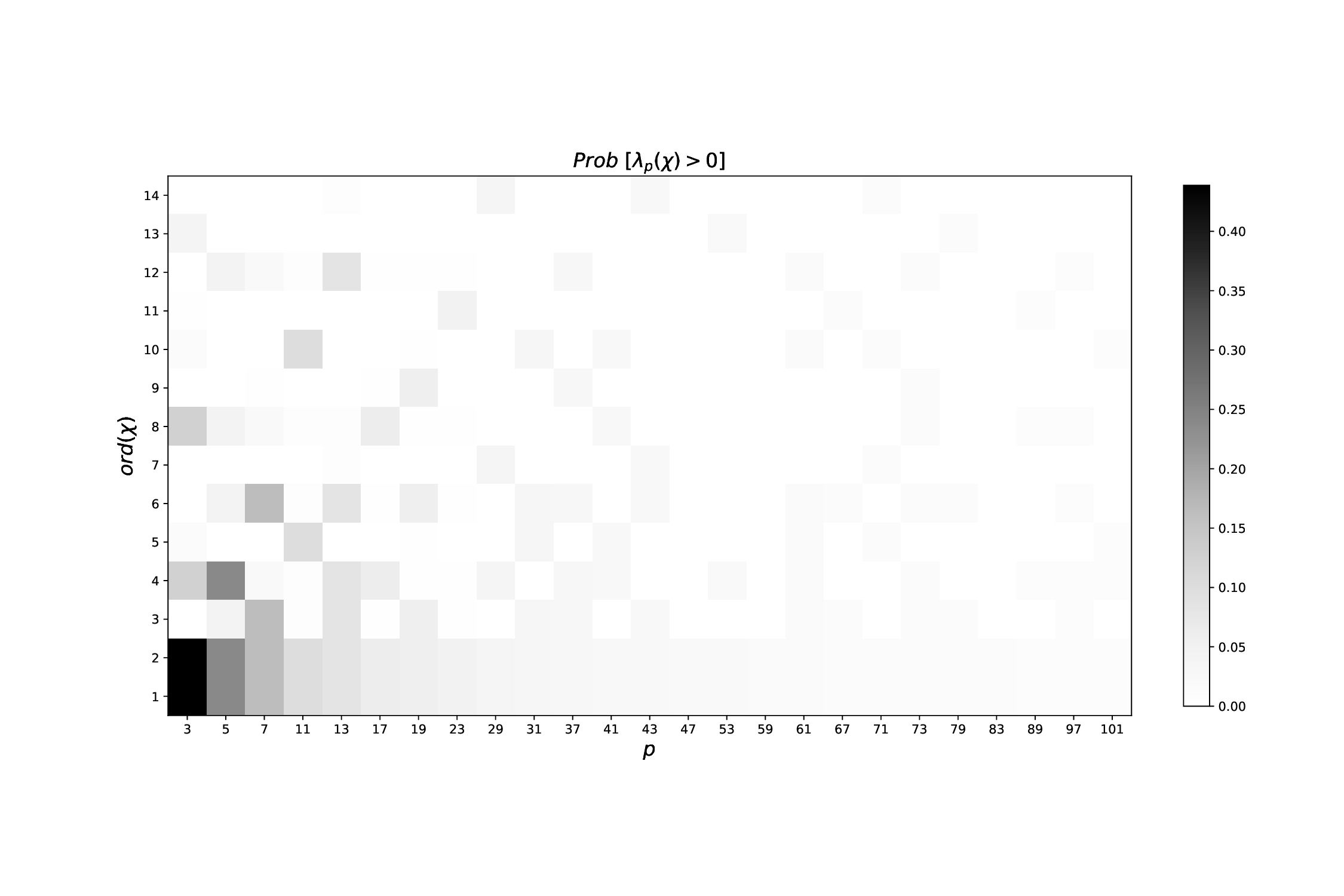}
  \caption{Predicted probabilities that $\lambda_p(\chi)>0$ for primes $p$ and characters $\chi$, under the assumption that $p$ is coprime to $\ord(\chi)$.}
  \label{lambdaprob}
\end{figure}

\vspace{-4mm}

Throughout a prime $p$ is called {\it $\chi$-regular} if $\lambda_p(\chi\omega^{i})=0$ at all $i\in\{0,\dots,p-2\}$ such that $\chi \omega^i$ is even, unless $\chi \omega^{i}(p)=1$
in which case we use $\lambda_p^{\text{\rm corr}}(\chi\omega^{i})=0$ instead.
From the interpolation rule in Equation (\ref{InterpFormula}) for even $\chi$, a prime $p$ is $\chi$-regular if and only if
$p$ does not divide the numerators of the numbers $B_{2, \chi},B_{4, \chi},\dots, B_{p-1, \chi}$.
Assuming Conjecture \ref{conj1} holds for a character $\chi$, we derive an asymptotic formula (Theorem \ref{TotThm}) that
allows us to estimate the exact proportion of  $\chi$-regular primes.

\begin{conj}\label{conj2}
For a fixed character $\chi$, the proportion of $\chi$-regular primes is equal to
$1+\frac{e^{-\frac{1}{2}}-1}{\varphi(\ord(\chi))}$,
while the proportion of $\chi$-irregular primes equals
$\frac{1-e^{-\frac{1}{2}}}{\varphi(\ord(\chi))}$.
\end{conj}

Of course, if $\chi$ is trivial then the predicted proportion will be $e^{-\frac{1}{2}}\approx 0.6065$ which coincides with the
standard predicted proportion of regular primes over $\QQ$.
However, note that our derivation is based on Conjecture \ref{conj1}, which refines the usual heuristics that the
Bernoulli numbers $\big\{B_i\big\}_{i\in2\mathbb{N}}$ are uniformly random modulo $p$.
The relationship between $\chi$-regularity and $\ord(\chi)$ is represented in Figure \ref{propregular} below,
and this prediction is reflected by the numerical data that we collect in Section \ref{Sect3}.

\vspace{-6mm}

\begin{figure}[H]
  \includegraphics[width=\linewidth]{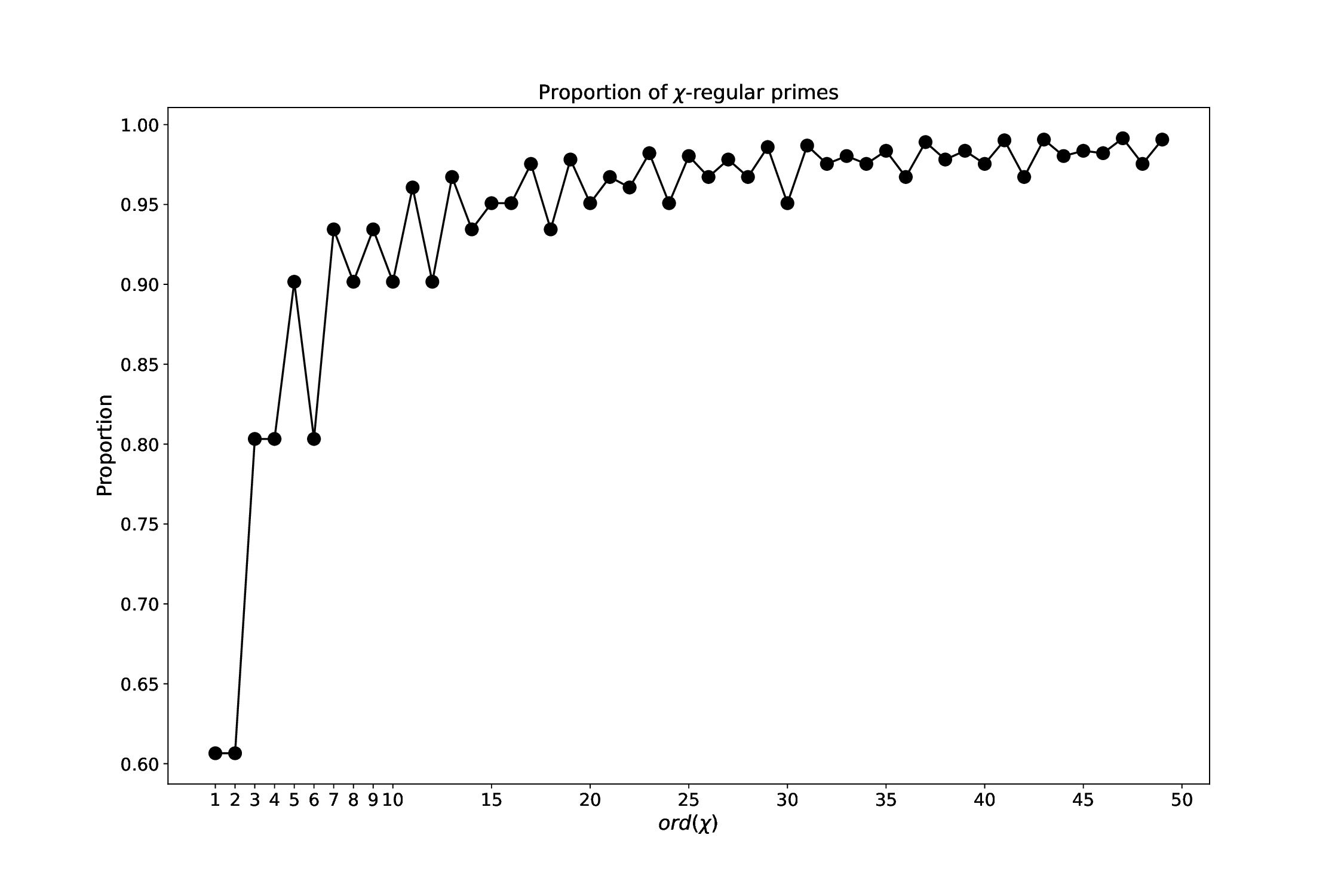}
  \caption{The predicted proportion of $\chi$-regular primes versus $\ord(\chi)$.}
  \label{propregular}
\end{figure}

\vspace{-4mm}

Let $F$ be a totally real abelian extension of $\QQ$ such that $p$ is unramified in $F/\QQ$. One says that $p$ is {\em $F$-regular} if it is $\chi$-regular for
every character $\chi$ of $\text{Gal}(F/\QQ)$.
Applying the analytic class number formula and the construction of $p$-adic $L$-functions,
$p$ is $F$-regular iff $p$ does not divide the minus part of the class numbers of $F(\mu_{p^n})$ for all $n\in \mathbb{N}$ (see \cite[Theorems 4.17 and 7.13]{Wa2}).

\begin{conj}\label{conj3} Fix a prime $p\gg 0$ and an abelian group $G$ of order $m$, isomorphic to a sum of cyclic groups of order $m_i$ with $i\in\{1,\dots, n\}$.
For totally real abelian extensions $F/\mathbb{Q}$ with Galois group $G$ and $p\nmid m\cdot\text{\rm disc}_F$, one estimates that
$$
\text{\rm Prob}\big(\text{\rm $p$ is $F$-regular}\big) \;\approx\; \exp\left(\!-\frac{\prod_{i=1}^n \gcd(m_i , p-1)}{2} \right) .
$$
\end{conj}

Recall that a character $\chi$ takes values inside $\QQ(\chi)=\QQ\big(\text{\rm Im}(\chi)\big)=\QQ(\mu_{\ord(\chi)})$.
We multiply the probabilities that $\chi$ is regular over all characters $\chi$ of $G$. 
For each cyclic component of order $m_i$, there are $\gcd(m_i,p-1)$ characters $\chi_i$ such that $p \equiv 1 \;\!\!\!\!\ \mod  \ord(\chi_i)$.
For these characters, the probability of being regular is approximately $e^{-1/2}$, while for the other characters, the probability is close to $1$. 
So we obtain $\prod_{i=1}^n \gcd(m_i , p-1)$ factors $e^{-1/2}$.

For $F=\QQ$, we recover the usual conjecture on the proportion of regular primes.
For real quadratic fields $F$ with non-trivial character $\chi$, a prime $p$ is $F$-regular iff $\lambda_p(\omega^ i) = 0$ and
$\lambda_p(\chi \omega^ i) = 0$ for all $i\in\{0,2,\dots,p-3\}$: the predicted probability is $(e^{-1/2})^2 = e^{-1}$.
For cyclic cubic fields $F$ with non-trivial character $\chi$, a prime $p$ is $F$-regular precisely when $\lambda_p(\chi^j \omega^ i) = 0$ for all $i\in\{0,2,\dots,p-3\}$ and $j=0,1,2$:
if $3\;\big|\; p-1$ then one predicts a probability of $e^{-3/2}$, otherwise one predicts $e^{-1/2}$.

\section{Estimates from $p$-adic random matrix theory}\label{Sect2}

\noindent
To study the behaviour of $\lambda_p(\chi)$, we utilise the Iwasawa Main Conjecture to relate this invariant to the action of a certain operator `$\gamma_0$' on an arithmetic module $W$.
Let $F$ be a totally real field, and assume that the prime $p\neq 2$ is unramified in $F$.
We write $F_{\infty}$ for the cyclotomic $\mathbb{Z}_p$-extension of $F$, so that $\Gamma
:=\text{\rm Gal}(F_{\infty}/F)\cong\mathbb{Z}_p$. Moreover
if $\mu_{p^{\infty}}\!\subset\overline{F}$ is the group of $p$-power roots, then
$\text{\rm Gal}\big(F(\mu_{p^{\infty}})/F\big) \cong \Gamma \times \mathbb{F}_p^{\times}$.
Wiles \cite{Wi} proved for {\it even} Dirichlet characters $\chi$ whose restriction to $\Gamma$ is trivial,
$$
\mathcal{G}_{\chi}\big(\gamma_0-1\big) \;=\; \text{\rm char}_{\Lambda_{\mathcal{O}}}\!\left(W^{(\psi)}\right)
\quad\text{\rm up to a unit, \; where $\psi=\chi^{-1}\omega$.}
$$
Here $W=\text{\rm Gal}(M_{\infty}/F_{\infty})$
is the Galois group of the maximal abelian $p$-extension $M_{\infty}$ of $F(\mu_{p^{\infty}})$ unramified
outside $p$, $W^{(\psi)}$ denotes its $\psi$-eigenspace, and $\Gamma=\gamma_0^{\mathbb{Z}_p}$.
It follows that zeroes of the $\chi$-twisted $p$-adic zeta-function $\mathcal{G}_{\chi}(T)$
correspond (with multiplicity) to the zeroes of the characteristic polynomial of $\gamma_0\in\!\Lambda$
acting on $W^{(\psi)}\!$. We model the latter situation using the random matrix approach of \cite[Section 4]{EJV}.

\subsection{The statistics for $\text{\rm GL}(n,\mathcal{O})$  and $\text{\rm GSp}(2n,\mathcal{O})$}
\label{Sect2.1}
Henceforth we assume $F/\mathbb{Q}$ is an abelian extension.
Since $F\cap \mathbb{Q}(\mu_{p^{\infty}})=\mathbb{Q}$ as $p\nmid \text{\rm disc}(F)$, one must have
$$
\text{\rm Gal}\big(F(\mu_{p^{\infty}})/\mathbb{Q}\big)
\;\cong\; \Gamma \times \mathbb{F}_p^{\times}  \times \text{\rm Gal}(F/\mathbb{Q}) .
$$
Note that $W^{(\psi)}:=\big\{w\in \text{\rm Gal}(M_{\infty}/F_{\infty})\otimes_{\mathbb{Z}_p}
\mathcal{O}_{\chi}\;\big|\; w^g=\psi(g)\cdot w\big\}$
is a compact $\Lambda\otimes_{\mathbb{Z}_p}\!\mathcal{O}_{\chi}$-torsion module of finite type, where $\mathcal{O}_{\chi}$ is generated over $\mathbb{Z}_p$ by the values of the character $\chi$.
Throughout this section we will suppose that $\mathcal{O}=\mathcal{O}_{\chi}$
has a chosen uniformiser $\varpi$, with finite residue field $\mathcal{O}/\varpi \cong \mathbb{F}_q$ such that $q=p^f$.

\begin{nots*}
(i) We write $\text{\rm Mat}_{n\times n}(\mathcal{O})$ for the ring of $n\times n$ matrices over $\mathcal{O}$, and $\text{\rm GL}(n,\mathcal{O})\subset\text{\rm Mat}_{n\times n}(\mathcal{O})$
denotes the multiplicative group of invertible matrices.

\smallskip
(ii) The generalised symplectic group $\text{\rm GSp}(2n,\mathcal{O})=\bigcup_{\alpha\in\mathcal{O}^{\times}}\text{\rm GSp}_{\alpha}(2n,\mathcal{O})$ by definition consists of the disjoint matrix subsets
$$
\text{\rm GSp}_{\alpha}(2n,\mathcal{O})\;=\;
\left\{A\in \text{\rm GL}(n,\mathcal{O}) \;\Big|\; \big\langle A\underline{x}, A\underline{y} \big\rangle
= \alpha\cdot \big\langle \underline{x} , \underline{y} \big\rangle
\right\}
$$
where the pairing $\big\langle -,- \big\rangle:\mathcal{O}^{\oplus n}\oplus \mathcal{O}^{\oplus n}\rightarrow\mathcal{O}$ sends $\big(\underline{x} , \underline{y} \big)\mapsto\sum_{i=1}^n x_iy_{i+n}-x_{i+n}y_i$.
\end{nots*}

Let $A$ be a
matrix selected from either of the groups $\text{\rm GL}(n,\mathcal{O})$ or $\text{\rm GSp}(2n,\mathcal{O})$.
We may consider the characteristic polynomial of the matrix $A-I_n$, namely
$$
f_{A-I_n} (T) \;:=\; \det\big(T\cdot I_n - (A-I_n)\big) \;=\; \det\big((1+T)\cdot I_n - A\big) \;\in\; \mathcal{O}[T].
$$
One can view $A$ as acting on the free $\mathcal{O}$-module $M=\mathcal{O}^{\oplus n}$:
there is a {\em Fitting decomposition}
$M=M^{\text{\rm uni}}\oplus M^{\dag}$ where $A-I_n$ operates on $M^{\text{\rm uni}}$ topologically nilpotently, and on $M^{\dag}$ invertibly.
We have a corresponding factorisation $f_{A-I_n}(T) =  P_A(T) \cdot u(T)$ where
$P_A(T)$ is a distinguished polynomial -- called the {\em associated polynomial} -- and $u(T)$ satisfies
$u(0) \neq 0 \mod \varpi$ (here $P_A(T)=1$ iff $A-I_n$ is invertible over $\mathcal{O}$).
If one interprets $M^{\text{\rm uni}}$ as the sum of all generalized eigenspaces $M_{\lambda}$
for eigenvalues $\lambda\in\overline{\QQ}_p$ of $A$ satisfying $|\lambda-1|_p<1$, it follows that
$P_A(T)=\prod_{\lambda}(T+1-\lambda)$
where the product is taken over all eigenvalues $\lambda$ of $A$ which are $p$-adically close to one.

Our intended goal is to compute the probability that a uniform random matrix in either one of the above groups has associated polynomial of a given degree $r$.
For reasons of space only, we limit ourselves to doing the calculation for $\text{\rm GL}(n,\mathcal{O})$.

\begin{dfn}\label{rhoDef} If $q=p^f$ and the integer $r\geq 0$, then we define
$$
\rho(q,r)\;:=\;q^{-r}\times\;\prod_{t>r}\;\big(1-q^{-t}\big) .
$$
\end{dfn}

\begin{thm}\label{DistThm} The proportion of randomly chosen matrices $A \in \text{\rm GL}(n,\mathcal{O})$ whose associated polynomial $P_A(T)$ has degree $r$ tends towards
$\rho(q,r)$ as $n \rightarrow \infty$.
\end{thm}

\begin{proof}
The degree of the associated distinguished polynomial $P_A(T)$ only depends on $A \mod \varpi$, so we reduce the task to a counting problem
involving matrices over $\mathbb{F}_q$. For example, if $r=0$ then the proportion of matrices $A$ in $\text{\rm GL}(n,\mathcal{O})$
such that $A-I_n$ is invertible is the number of $n \times n$ matrices over $\mathbb{F}_q$ whose eigenvalues lie strictly outside $\{0,1\}$, divided through by
the size of $\text{\rm GL}(n,\mathbb{F}_q)$.

Let us choose an $r \in \{0,1,\dots,n\}$. Then the number of matrices $A \in \text{\rm GL}(n,\mathbb{F}_q)$ such that $\text{\rm deg}(P_A(T))=r$ is given by the product of:
\begin{enumerate}
\item the number of unipotent $r \times r$ matrices over $\mathbb{F}_q$;
\item the number of $(n-r) \times (n-r)$ matrices $A$ such that $A$ and $A-I_{n-r}$ are invertible; and
\item  the number of decompositions of $\mathbb{F}_q^n$ into an $r$-dimensional subspace and a complementary $(n-r)$-dimensional subspace.
\end{enumerate}
In fact each matrix $A \in \text{\rm GL}(n,\mathbb{F}_q)$ is uniquely determined by the associated Fitting decomposition over $\mathbb{F}_q$:
the degree of $P_A(T)$ is $r$ if and only if the subspace on which $A-I_n$ operates nilpotently (or equivalently, $A$ acts unipotently) is $r$-dimensional.

It is a straightforward exercise to show that $\#\text{\rm GL}(r,\mathbb{F}_q) = q^{r^2} \prod_{1\leq i \leq r} \; (1-q^{-i})$.
We next claim that the cardinalities of the three sets above are respectively:
\begin{enumerate}
\item $q^{r^2 - r}$
\item $ \displaystyle\left(  1\;+\; \sum_{i=1}^{n-r} \frac{(-1)^i}{(q-1)(q^2-1) \dots (q^{i} - 1)} \right) \times  \#\text{\rm GL}(n-r,\mathbb{F}_q) $
\item $\displaystyle\frac{ \#\text{\rm GL}(n,\mathbb{F}_q)}  {\#\text{\rm GL}(r,\mathbb{F}_q) \times  \#\text{\rm GL}(n-r,\mathbb{F}_q) }$
\end{enumerate}
The formula for (1) is well known, and (3) is quite obvious from the definition. However for (2), we are looking at the number of invertible matrices not having $1$ as an eigenvalue.
In \cite{FrWa,Wa} the proportion of matrices in $\text{\rm GL}(n,\mathbb{F}_p)$ having $1$ as an eigenvalue was computed, and as the proof works over $\mathbb{F}_q$,
the formula follows.

Multiplying together the above cardinalities, we may thereby conclude that the proportion of matrices $A$ in $\text{\rm GL}(n,\mathcal{O})$ such that $\text{\rm deg}(P_A(T)) = r$
must be equal to
$$ \frac{q^{-r}} {\prod_{1 \leq i\leq r} \; (1-q^{-i})} \times  \left(  1\;+\; \sum_{i=1}^{n-r} \frac{(-1)^i}{(q-1)(q^2-1) \dots (q^{i} - 1)} \right) . $$
The first factor counts the exact proportion of unipotent matrices in $\text{\rm GL}(r,\mathcal{O})$, while the initial couple of terms in the second factor will sum up to
equal $1 - \frac{1}{q-1}=\frac{q-2}{q-1}$.
The rest of the right-hand series converges very fast, so the proportion mainly depends on $q$ and $r$, and the dimension of the $\mathcal{O}$-module seems to be less important.
Taking the limit as $n \rightarrow \infty$, the theorem will follow provided one can show that
$$ \lim_{n \rightarrow \infty}  \left( 1+ \sum_{i=1}^{n-r} \frac{(-1)^i}{(q-1)(q^2-1) \dots (q^{i} - 1)}\right) =\; \prod_{t=1}^{\infty} \big(1-q^{-t}\big)
\;=\; \rho(q,0).$$
This is established in \cite{Wa} for $q=p$, and the proof extends seamlessly to $q=p^f$.
\end{proof}

\begin{rem*} (a) If we instead replace the group $\text{\rm GL}(n,\mathcal{O})$ in Theorem \ref{DistThm} with $\text{\rm GSp}_{\alpha}(2n,\mathcal{O})$  then
the same limiting value of $\rho(q,r)$ occurs (see \cite[\S4.2]{EJV} for $q=p$). We note that Ellenberg et al.\ assume in \cite{EJV} that the multiplier $\alpha$ in the generalised symplectic case does not reduce to $1$ in the residue field $\mathbb{F}_q$. Otherwise, one obtains a different distribution, and this can be used to model the $p$-part of relative class groups in the case when the base field contains a $p$-th root of unity (see \cite{Gar, Mal}). We thank the reviewers for their comments on this point.

\smallskip
(b) Returning to the equality $\text{\rm char}_{\Lambda_{\mathcal{O}}}\!\left(W^{(\psi)}\right) = \mathcal{G}_{\chi}(\gamma_0-1)$ from Wiles' theorem,
the probability that $\mathcal{G}_{\chi}(T)$ has exactly $r$ zeroes on the open unit disk is the same as the probability that
the characteristic power series $\text{\rm char}_{\Lambda_{\mathcal{O}}}\!\left(W^{(\psi)}\right)$ has $r$ zeroes.

\smallskip
(c) In the random matrix model, $\mathcal{G}_{\chi}(T)$ represents the associated distinguished polynomial of a large invertible matrix over $\mathcal{O}$,
and $W^{(\psi)}$ represents the subspace (of a hypothetical large $\mathcal{O}$-module) on which $\gamma_0-1$ acts topologically nilpotently.
 \end{rem*}

\subsection{An application of Euler's pentagonal theorem}\label{Sect2.2}
If one believes that the statistical behaviour of $\lambda_p(\chi)$ is properly modelled through random matrix theory for either
$\text{\rm GL}(n,\mathcal{O})$
or $\text{\rm GSp}_{\alpha}(2n,\mathcal{O})$ as $n\rightarrow\infty$, one would then expect
$$
\text{\rm Prob}\big(\lambda_p(\chi)=r\big) \;\approx\; \rho\big(q,r\big)
\qquad\text{(see Theorem \ref{DistThm})}
$$
at each prime $p$ such that $\mathcal{O}_{\chi}=\mathbb{Z}_p[\text{\rm Im}(\chi)]$ has residue field $\mathbb{F}_{q}\cong\mathcal{O}_{\chi}/\varpi$ with $q=p^f$.
Indeed this is precisely Conjecture \ref{conj1} that was initially stated in the Introduction.
In this section, we derive some asymptotic formulae describing the behaviour of $\rho(q,r)$ as a function of $p\in\mathbb{P}$ (with $q=p^f$), which will be needed in our heuristics.

\begin{dfn} For each integer $m\geq 1$ and congruence class $a\in\big(\mathbb{Z}/m\mathbb{Z}\big)^{\times}\!$, one partitions the prime numbers $\mathbb{P}$ into subsets
$$
\mathbb{P}_{\!a(m)} \;:= \;
\left\{p\in\mathbb{P}\;\big|\;
p\equiv a\;(\!\!\!\!\!\!\mod m)\right\}
$$
so in particular, there is a disjoint union $\mathbb{P}=\coprod_{a(\!\!\!\!\mod m)} \mathbb{P}_{\!a(m)}$.
\end{dfn}

\begin{prop}\label{rhoProp} For $X \in \mathbb{R}$, as usual one sets $\pi(X):=\#\big\{p \in \mathbb{P}\ \big| \ p \leq X\big\}$.

\smallskip
 (i) At every integer $r\geq 0$ one has the sequence of bounds
$$
\qquad
0 \;<\; \rho(q,r)\;<\; q^{-r}  \;\leq\; 1 ,
$$
thence $\lim_{p\rightarrow\infty}\rho\big(p^f,r\big)=0$ if $r>0$
where $p$ ranges through the prime numbers.

\smallskip
\!\!(ii) Given a class $a\in\big(\mathbb{Z}/m\mathbb{Z}\big)^{\times}\!$ and a real number $X >0 $, we have the estimate
$$
\qquad\qquad
\sum_{p\in\mathbb{P}_{\!a(m)},\; p\leq X} \rho(p,r) \;\approx\;
\begin{cases}
\frac{1}{\varphi(m)}\big(\pi(X)-\log\log X\big)+O(1) & \text{if $r=0$} \\
\frac{\log\log X}{\varphi(m)}\times\big(1 + o(1)\big) & \text{if $r=1$} \\
\;O(1)  & \text{if $r\geq 2$.}
\end{cases}
$$
$$
\quad
\text{(iii) If the degree $f\geq 2$ then} \;\sum_{p\in\mathbb{P}_{\!a(m)},\; p\leq X} \rho\big(p^f,r\big) \;\approx\;
\begin{cases}
\frac{\pi(X)}{\varphi(m)}+O(1) & \text{if $r=0$} \\
\;O(1)  & \text{if $r\geq 1$.}
\end{cases}
$$
\end{prop}

\begin{proof} 
(i) is completely obvious.
To establish parts (ii) and (iii), we recall Euler's Pentagonal Formula from \cite{Eu}: for every $y\in\mathbb{R}$ with $|y|<1$,
$$
\prod_{t\geq 1}\;(1-y^t) \;=\; 1 \;+\; \sum_{k=1}^{\infty}\; (-1)^k\Big(y^{\frac{3k^2-k}{2}}+y^{\frac{3k^2+k}{2}}\Big) .
$$
Now substituting in $y=q^{-1}=p^{-f}$, one immediately deduces that
$$
\rho\big(p^f,0\big) =\;
\prod_{t\geq 1}\;(1-p^{-ft}) \;=\; 1 -p^{-f}-p^{-2f}+ \sum_{k=2}^{\infty} (-1)^k\Big(p^{-\frac{(3k^2-k)f}{2}}+p^{-\frac{(3k^2+k)f}{2}}\Big) .
$$
In particular, we have $\lim_{p\rightarrow\infty}\rho\big(p^f,0\big)=1$.
Summing over the primes $p$ congruent to $a\;(\!\!\!\!\mod m)$ and less than or equal to $X$:
$$
\sum_{p\in\mathbb{P}_{\!a(m)},\; p\leq X} \rho\big(p^f,0\big) \;\;\approx\;\;
\text{\rm dens}(\mathbb{P}_{\!a(m)})\cdot \pi(X) \;-\sum_{p\in\mathbb{P}_{\!a(m)},\; p\leq X} p^{-f} \;+ \sum_{p\in\mathbb{P}_{\!a(m)},\; p\leq X}  a(p^f)
$$
where the coefficients $a(p^f)=-p^{-2f}+ \sum_{k=2}^{\infty} (-1)^k\Big(p^{-\frac{(3k^2-k)f}{2}}+p^{-\frac{(3k^2+k)f}{2}}\Big)\in\mathbb{R}$,
and the notation $\text{\rm dens}(\mathcal{S})$ indicates the Dirichlet density of the set $\mathcal{S}$ inside $\mathbb{P}$.

\begin{rems*} (a) If $f=1$ then $\sum_{p\in\mathbb{P}_{\!a(m)},\; p\leq X}\; p^{-f}\approx \text{\rm dens}(\mathbb{P}_{\!a(m)})\cdot\log\log X + O(1)$
from \cite{Ap}, and if $f>1$ this sum is bounded above independently of the choice of $X$.

\smallskip
(b) The infinite series $\sum_{p\in\mathbb{P}_{\!a(m)}}  p^{-2f}$ will always converge; in fact $\sum_{p\in\mathbb{P}_{\!a(m)}}p^{-2f}$
is bounded above by $\wp(2f)$ where $\wp(s)=\sum_{p\in\mathbb{P}}\; p^{-s}$ is the prime zeta-function.

\smallskip
(c) If $\wp_{a(m)}(s)=\sum_{p\in\mathbb{P}_{\!a(m)}}\; p^{-s}$ denotes a Hurwitz-type prime zeta-function, the convergence of
$$
\sum_{p\in\mathbb{P}_{\!a(m)}}\;
\sum_{k=2}^{\infty} \; (-1)^k\Big(p^{-\frac{(3k^2-k)f}{2}}+p^{-\frac{(3k^2+k)f}{2}}\Big)
$$
to the explicit value $\sum_{k=2}^{\infty} \; (-1)^k\Big(\wp_{a(m)}\Big(\frac{(3k^2-k)f}{2}\Big)+\wp_{a(m)}\Big(\frac{(3k^2+k)f}{2}\Big)\Big)$ inside $\mathbb{R}$ follows from the absolute convergence of the double sum.
\end{rems*}

Combining together (a)-(c) above and also noting that $\text{\rm dens}(\mathbb{P}_{\!a(m)})=1/\varphi(m)$, the stated estimate for $\rho(q,r)$ follows  immediately in the situation where $r=0$.
Let us instead suppose that $r=1$. Then for any choice of $\epsilon>0$,
$$
p^{-f}\cdot\rho\big(p^f,0\big) \;<\;
\rho\big(p^f,1\big) = \frac{p^{-f}}{1-p^{-f}}\cdot\rho\big(p^f,0\big)  \;<\;
(1+\epsilon)\times p^{-f}\cdot \rho\big(p^f,0\big)
$$
provided that the prime $p>X_{\epsilon}:=\big(1+\frac{1}{\epsilon}\big)^{1/f}$ say.
From the previous estimates,
$$
\sum_{p\in\mathbb{P}_{\!a(m)},\; X_{\epsilon}<p\leq X} \rho\big(p^f,1\big) \;\;\lesssim\;\;
\left( 1+\epsilon  \right) 
\times \sum_{p\in\mathbb{P}_{\!a(m)},\; X_{\epsilon}<p\leq X} p^{-f} .
$$
If $f=1$ then $\sum_{p\in\mathbb{P}_{\!a(m)},\; p\leq X} \; p^{-f} \approx\frac{\log\log X}{\varphi(m)}+O(1)$,
otherwise the sum is bounded. As a direct consequence, one has
$
\sum_{p\in\mathbb{P}_{\!a(m)},\; p\leq X} \rho\big(p,1\big) \approx \Big(\frac{1+o(1)}{\varphi(m)}\Big)\cdot\log\log X
+O(1)
$
whilst if $f\geq 2$ then $\sum_{p\in\mathbb{P}_{\!a(m)},\; p\leq X} \;\rho\big(p^f,1 \big)$ is bounded above independently of $X$.

\noindent
Finally, let us suppose the integer $r\geq 2$. Then one has the strict upper bounds
$$
\rho\big(p^f,r\big) = \prod_{1\leq t\leq r}\frac{1}{1-p^{-ft}}\times p^{-fr}\cdot\rho\big(p^f,0\big)
\;\;<\;\; 2^r \times p^{-fr}\cdot\rho\big(p^f,0\big)
\;\;<\;\; \left(\frac{2}{p^f}\right)^r
$$
because $\rho\big(p^f,0\big)<1$. Clearly $\sum_{p\in\mathbb{P}_{\!a(m)},\; p\leq X}\left(\frac{2}{p^f}\right)^r$ is bounded
independently of $X$  hence the sum of these $\rho\big(p^f,r\big)$ is too, which completes all cases in (ii) and (iii).
\end{proof}

\begin{dfn} Given any integer $r\geq 0$, a real number $X>0$ and a Dirichlet character $\chi$, let us define
$$
\theta_{\leq X}^{(r)}(\chi) \;:=\; \#\left\{p\in\mathbb{P}\;\Big|\; p\leq X \;\text{and}\; \lambda_p^{\text{corr}}(\chi)=r
\right\}
$$
where (as before) one sets $\lambda_p^{\text{corr}}(\chi) := \lambda_p(\chi)$ in the situation for which $\chi \omega^{-1} (p) \neq 1$, and
otherwise sets $\lambda_p^{\text{corr}}(\chi) :=  \lambda_p(\chi)-1$ if $\chi \omega^{-1} (p) = 1$
(the trivial zero situation).
\end{dfn}

Equivalently, $\theta_{\leq X}^{(r)}(\chi)$ counts up precisely those primes less than or equal to $X$ whose $\chi$-twisted $\lambda$-invariant is equal to $r$.
Putting $m=\ord(\chi)$ we may partition the primes into
$\mathbb{P}=\coprod_{a(\!\!\!\!\mod m)} \mathbb{P}_{\!a(m)}$, which yields the decomposition
$$
\theta_{\leq X}^{(r)}(\chi)
\;\;=\;
\sum_{a\in(\mathbb{Z}/m\mathbb{Z})^{\times}} \;\sum_{p\in\mathbb{P}_{\!a(m)},\; p\leq X}  
\delta \big(\lambda_p^{\text{corr}}(\chi)=r\big)
$$
where $\delta(\text{\rm TRUE})=1$ and $\delta(\text{\rm FALSE})=0$.

For example, if $a=1$ then $\mathbb{P}_{\!1(m)}$ are the primes $p$
congruent to $1$ modulo $\ord(\chi)$ i.e. exactly those prime numbers $p$ such that
$\mathcal{O}_{\chi}=\mathbb{Z}_p[\text{\rm Im}(\chi)]$ has residue field $\mathbb{F}_p$.
Given the random matrix model is accurate, the above sum is approximated by
$$
\theta_{\leq X}^{(r)}(\chi) \;\;\approx\;  \sum_{p\in\mathbb{P}_{\!1(m)},\; p\leq X}   \rho\big(p,r\big)
\;\;+\!\!\sum_{\tiny\begin{array}{c} a=2, \\ \gcd(a,m)=1\end{array}}^m \!\sum_{p\in\mathbb{P}_{\!a(m)},\; p\leq X} \rho\big(\#(\mathcal{O}_{\chi}/\varpi),r\big) .
$$
As we shall shortly see, as an isolated function of $r$ (with $X$ fixed) the summation attains its largest values when $r=0$, and tends towards zero as $r$ becomes large.
The terms on the right-hand side of this approximation can be estimated using Proposition \ref{rhoProp}, in the following manner:
\begin{itemize}
\item if $r=0$ then the first sum behaves like $\frac{\pi(X)-\log\log X}{\varphi(m)}+O(1)$  while the second sum behaves like $\big(1-\frac{1}{\varphi(m)}\big)\cdot\pi(X)+O(1)$,
as the value $X\rightarrow\infty$;

\item if $r=1$, the first sum is $\frac{\log\log X}{\varphi(m)}\cdot\big(1 + o(1)\big)$
and the second sum is $O(1)$;

\item if $r\geq 2$ then both of these sums will be bounded independently of $X>0$.
\end{itemize}
Noting the various predictions arising from this model, one should therefore expect
$$
\theta_{\leq X}^{(r)}(\chi) \;\;\approx\;\;
\begin{cases}
\pi(X)-\frac{\log\log X}{\varphi(m)}+O(1) & \text{if $r=0$} \\
\frac{\log\log X}{\varphi(m)}+o\big(\log\log X\big) & \text{if $r=1$} \\
\;O(1)  & \text{if $r\geq 2$.}
\end{cases}
$$
Lastly we recall that the prime number theorem implies $\pi(X)\approx \frac{X}{\log(X)}$ for $X\gg 0$.

\begin{conj} For increasing $X$, one predicts the asymptotic behaviour:
\begin{eqnarray*}
& & \;\text{(i)} \;\;\; \lim_{X\rightarrow\infty } \frac{\theta_{\leq X}^{(0)}(\chi)}{X/\log(X)} \;=\; 1 \\
& & \text{(ii)} \;\;\; \lim_{X\rightarrow\infty } \frac{\theta_{\leq X}^{(1)}(\chi)}{\log\log(X)} \;=\; \frac{1}{\varphi\big(\ord(\chi)\big)} \\
& & \text{(iii)} \;\;\; \theta_{\leq X}^{(r)}(\chi)\;\;\text{is bounded above if $r\geq 2$.}
\qquad\qquad\qquad\qquad\qquad\qquad
\end{eqnarray*}
\end{conj}

\subsection{Heuristics for the $\chi$-regularity of $p$}\label{Sect2.3}
Recall that a prime number $p$ is called regular if it does not divide the class number of the cyclotomic field $\mathbb{Q}(\mu_p)$. Indeed it
is essentially a consequence of Kummer's criterion that a prime $p$ is regular if and only if each power series $\mathcal{G}_{\!\omega^{2j}}(T)\in\mathbb{Z}_p[\![T]\!]$
is a unit of the Iwasawa algebra.

\begin{dfn} (i) For any Dirichlet character $\chi$, we shall define the {\it total $\lambda$-invariant} associated to $\chi$ by
$$
\qquad
\lambda_p^{\text{\rm tot}}(\chi) \; := \sum_{\substack{j=0 \\ \chi\omega^j \text{even}}}^{p-2} \lambda_p^{\text{corr}}(\chi\omega^{j})
\;\in\; \mathbb{Z}_{\geq 0} .
$$
\qquad
(ii) A prime $p$ is called {\it $\chi$-regular}\footnote{This differs in general from Ernvall's definition of $\chi$-regularity for a prime $p$, as he requires
that $\lambda_p(\chi^{\sigma}\omega^{j})=0$
for all even branches and elements $\sigma\in\text{\rm Gal}(\overline{\mathbb{Q}}/\mathbb{Q})$, which is a stronger condition.}
if $\lambda_p^{\text{corr}}(\chi\omega^{j})=0$ for every $j\in\{0,\dots,p-2\}$ such that $\chi\omega^{j}$ is even;
otherwise it is called a {\it $\chi$-irregular} prime, not surprisingly.
\end{dfn}

Applying the interpolation formula (\ref{InterpFormula}) for the $\chi$-twisted $p$-adic zeta-function,
if $\chi$ is even then it follows  that a prime $p$ is $\chi$-regular if and only if $p$
does not divide the numerator of any of the $\chi$-twisted Bernoulli numbers $B_{2,\chi},\dots,B_{p-1,\chi}$.
Alternatively, if $\chi$ is an odd Dirichlet character then $p$ will be $\chi$-regular if and only if $p$
does not divide the numerator of any of the $\chi$-twisted numbers $B_{1,\chi},\dots,B_{p-2,\chi}$.
Now assuming our $p$-adic random matrix models are accurate,
the probability that $\lambda_p(\chi\omega^{j})=0$ should be independent of the branch $\omega^{j}$.
This independence implies
$$
\text{\rm Prob}\big(\lambda_p^{\text{\rm tot}}(\chi)=0\big) \;\;\approx\;\;\prod_{\substack{j=0 \\ \chi \omega^j \text{ even} }}^{p-2} \;
\text{\rm Prob}\big(\lambda^{\text{corr}}_p(\chi\omega^{j})=0\big) \;\;\approx\;\; \rho(q,0)^{\frac{p-1}{2}}
$$
where the residue field of $\mathbb{Z}_p[\text{\rm Im}(\chi)]$ has $q=p^f$ elements,
thus one should expect
$$
\pi(X)^{-1}\;\times\sum_{p\in\mathbb{P},\;p\leq X} \; \rho\big(\#(\mathcal{O}_{\chi}/\varpi),0\big)^{\frac{p-1}{2}}
$$
to tend towards the true proportion of $\chi$-regular primes as the bound $X$ gets larger.

\begin{thm}\label{TotThm} Given a real number $X\gg 0$, one has the approximation
$$
\sum_{p\in\mathbb{P},\;p\leq X} \rho\big(\#(\mathcal{O}_{\chi}/\varpi),0\big)^{\frac{p-1}{2}}\approx\;
\pi(X)\times\left(\frac{\varphi(\ord(\chi))+e^{-1/2}-1}{\varphi(\ord(\chi))}\right) +\; O\big(1 \big) .
$$
\end{thm}

\begin{proof}
We begin with some simple estimates. If one has $y\in (0,1)$ and $p>3$
then $1-\left(\frac{p-1}{2}\right)y < (1-y)^{\frac{p-1}{2}} < 1-y$,
in which case
$$
1-p^{-(t-1)f} \;<\;
1-\left(\frac{p-1}{2}\right)p^{-tf} \;<\; \big(1-p^{-tf}\big)^{\frac{p-1}{2}} \;<\; 1-p^{-tf} .
$$
Taking a product over $t$:\;
$
\prod_{t\geq 2} \big(1-p^{-tf}\big)
< \prod_{t\geq 3} \big(1-p^{-tf}\big)^{\frac{p-1}{2}}
< \prod_{t\geq 3} \big(1-p^{-tf}\big) .
$
If we write $\beta(y)=(1-y^{-1})(1-y^{-2})$, then by Definition \ref{rhoDef} one has the identity
$$
\prod_{t\geq 3} \big(1-p^{-tf}\big)^{\frac{p-1}{2}}
=\left(\frac{ \rho(p^f,0)}{(1-p^{-f})(1-p^{-2f})}\right)^{\frac{p-1}{2}}
\!=\; \beta\big(p^f\big)^{\frac{1-p}{2}}\!\cdot \rho\big(p^f,0\big)^{\frac{p-1}{2}} .
$$
Plugging this expression into our inequalities, one obtains the strict bounds
\begin{equation}\label{EqBnds}
\;\beta\big(p^f\big)^{\frac{p-1}{2}}\!\cdot \prod_{t\geq 2} \big(1-p^{-tf}\big)
\;\;<\;\; \rho\big(p^f,0\big)^{\frac{p-1}{2}}
\;\;<\;\; \beta\big(p^f\big)^{\frac{p-1}{2}}\!\cdot  \prod_{t\geq 3} \big(1-p^{-tf}\big) .
\end{equation}
It is well known that $\lim_{n\rightarrow\infty}\big(1+x/n\big)^{n}=\exp(x)$, so for $x=-1$ we deduce that
$$
\lim_{p\rightarrow\infty}\beta(p)^{\frac{p-1}{2}} =\;
\sqrt{\lim_{p\rightarrow\infty}\!\left(1-\frac{1}{p}\right)^p \times \lim_{p\rightarrow\infty} \!\frac{\big(1-\frac{1}{p^2}\big)^{p-1}}{\big(1-\frac{1}{p}\big)}}
\;=\; \sqrt{\exp\left(-1\right)\!\times 1 } \;=\; e^{-1/2}.
$$
Exploiting the algebra of limits, if the exponent $f\geq 2$ then $\lim_{p\rightarrow\infty}\beta\big(p^f\big)^{\frac{p-1}{2}}=1$.
Now fix any $\epsilon>0$ and write $m=\ord(\chi)$. Then there exists $N_{\epsilon}\in\mathbb{N}$ such that
\begin{itemize}
\item
$\Big|\beta\big(\#(\mathcal{O}_{\chi}/\varpi)\big)^{\frac{p-1}{2}}-e^{-1/2}\Big|<\epsilon$ for all $p\in\mathbb{P}_{\!1(m)}$ with $p\geq N_{\epsilon}$, \quad and

\smallskip
\item
$\Big|\beta\big(\#(\mathcal{O}_{\chi}/\varpi)\big)^{\frac{p-1}{2}}-1\Big|<\epsilon$ for all $p\in\mathbb{P}-\mathbb{P}_{\!1(m)}$ with $p\geq N_{\epsilon}$.
\end{itemize}
Firstly, the bounds in (\ref{EqBnds}) and our computation of $\lim_{p\rightarrow\infty}\beta(p)^{\frac{p-1}{2}}$ imply that
\begin{eqnarray*}
& & \big(e^{-1/2}-\epsilon\big)\times\!\!\! \sum_{\tiny\begin{array}{c} p\in\mathbb{P}_{\!1(m)}, \\ N_{\epsilon}\leq p\leq X\end{array}}
\prod_{t\geq 2} \big(1-p^{-t}\big) \\
& & \;\;<\!
\sum_{\tiny\begin{array}{c} p\in\mathbb{P}_{\!1(m)}, \\ N_{\epsilon}\leq p\leq X\end{array}} \rho\big(p,0\big)^{\frac{p-1}{2}}
\;\;<\; \;
\big(e^{-1/2}+\epsilon\big)\times\!\!\! \sum_{\tiny\begin{array}{c} p\in\mathbb{P}_{\!1(m)}, \\ N_{\epsilon}\leq p\leq X\end{array}} \prod_{t\geq 3} \big(1-p^{-t}\big) .
\end{eqnarray*}
Then it is an exercise involving Euler's pentagonal theorem to show
$$
\sum_{\tiny\begin{array}{c} p\in\mathbb{P}_{\!1(m)}, \\ N_{\epsilon}\leq p\leq X\end{array}}\prod_{t\geq d} \big(1-p^{-t}\big)
\;\approx\;\; \frac{\pi(X)}{\varphi(m)} \;+\;O(1)
\quad\text{with $d=2$ or $3$;}
$$
consequently $\sum_{p\in\mathbb{P}_{\!1(m)},\;p\leq X} \;\;\rho\big(p,0\big)^{\frac{p-1}{2}}\approx\; e^{-1/2}\cdot\frac{\pi(X)}{\varphi(m)}+O(1)$
for $X\gg 0$.

Secondly, if $q=p^f=\#(\mathcal{O}_{\chi}/\varpi)$ with $f\geq 2$ as before, then
again the bounds in (\ref{EqBnds}) and our computation of $\lim_{p\rightarrow\infty}\beta(p^f)^{\frac{p-1}{2}}$ together imply
\begin{eqnarray*}
& & \big(1-\epsilon\big)\times\!\!\! \sum_{\tiny\begin{array}{c} p\in\mathbb{P}-\mathbb{P}_{\!1(m)}, \\ N_{\epsilon}\leq p\leq X\end{array}}
\prod_{t\geq 2} \big(1-p^{-tf}\big) \\
& & \;\;<\!
\sum_{\tiny\begin{array}{c} p\in\mathbb{P}-\mathbb{P}_{\!1(m)}, \\ N_{\epsilon}\leq p\leq X\end{array}} \rho\big(p^f,0\big)^{\frac{p-1}{2}}
\;\;<\; \;
\big(1+\epsilon\big)\times\!\!\! \sum_{\tiny\begin{array}{c} p\in\mathbb{P}-\mathbb{P}_{\!1(m)}, \\ N_{\epsilon}\leq p\leq X\end{array}} \prod_{t\geq 3} \big(1-p^{-tf}\big) .
\end{eqnarray*}
Another exercise involving Euler's pentagonal theorem shows that for $d=2$ or $3$,
$$
\sum_{\tiny\begin{array}{c} p\in\mathbb{P}-\mathbb{P}_{\!1(m)}, \\ N_{\epsilon}\leq p\leq X\end{array}}\prod_{t\geq d} \big(1-p^{-tf}\big)
\;\approx\;\; \frac{(\varphi(m)-1)\cdot \pi(X) }{\varphi(m)} \;+\;O(1) ;
$$
it follows that $\sum_{p\in\mathbb{P}-\mathbb{P}_{\!1(m)},\;p\leq X} \;\rho\big(p^f,0\big)^{\frac{p-1}{2}}\approx\; \left(1-\frac{1}{\varphi(m)}\right)\cdot\pi(X)+O(1)$.

\noindent
Finally, one can always decompose
$$
\sum_{p\in\mathbb{P},\;p\leq X} \rho\big(\#(\mathcal{O}_{\chi}/\varpi),0\big)^{\frac{p-1}{2}} \;=\;
\sum_{p\in\mathbb{P}_{\!1(m)},\;p\leq X} \rho\big(p,0\big)^{\frac{p-1}{2}}
\; +\!\!
\sum_{p\in\mathbb{P}-\mathbb{P}_{\!1(m)},\;p\leq X} \rho\big(p^f,0\big)^{\frac{p-1}{2}}
$$
and the first summation behaves like $e^{-1/2}\cdot\frac{\pi(X)}{\varphi(m)}+O(1)$
as we just found, whilst the second sum behaves like $\left(1-\frac{1}{\varphi(m)}\right)\cdot\pi(X)+O(1)$
with $X\gg 0$.
One thereby concludes that
$$
\sum_{p\in\mathbb{P},\;p\leq X} \rho\big(\#(\mathcal{O}_{\chi}/\varpi),0\big)^{\frac{p-1}{2}} \approx\;
e^{-1/2}\cdot\frac{\pi(X)}{\varphi(m)} +
\left(1-\frac{1}{\varphi(m)}\right)\cdot\pi(X)+O(1)
$$
and the result is now fully established.
\end{proof}

\begin{rems*}
(i) A straightforward consequence of Theorem \ref{TotThm} is that
$$
\lim_{X\rightarrow\infty} \left(\frac{\sum_{p\in\mathbb{P},\;p\leq X} \;\rho\big(\#(\mathcal{O}_{\chi}/\varpi),0\big)^{\frac{p-1}{2}}}{\pi(X)}\right)
\;=\;\; 1+\frac{e^{-\frac{1}{2}}-1}{\varphi(\ord(\chi))}
$$
which was the principal motivation for making Conjecture \ref{conj2} in the first place.

\smallskip
(ii) If $\chi$ is trivial or quadratic then $\varphi(\ord(\chi))=1$, hence $1+\frac{e^{-\frac{1}{2}}-1}{\varphi(\ord(\chi))} =
e^{-\frac{1}{2}}$ which is consistent with the proportion of regular primes observed in nature.

\smallskip
(iii) As Im$(\chi)$ grows larger so does $\varphi(\ord(\chi))$, thus one expects the proportion
of $\chi$-regular primes to approach $1$ and that the $\chi$-irregular primes become sparser.
\end{rems*}

\subsection{Heuristics for totally real abelian extensions}\label{Sect2.4}
Let $F$ be an arbitrary totally real number field (not necessarily an abelian extension of $\mathbb{Q}$),
and choose a prime $p\neq 2$ that is unramified in $F$. Deligne and Ribet \cite {DeRi} associated a $p$-adic $L$-function, $\zeta_{p,F}(s)$,
whose $\omega^{2j}$-branches interpolate special values of the complex Dedekind zeta-function $\zeta_F(s,\omega^{2j})$. Let $\mathcal{D}_{F,\omega^{2j}}(T)$
be the associated power series.

\begin{dfn} (i) We say that a prime $p$ is {\it $F$-regular} if $p$ is unramified in $F$ and secondly
$\lambda\big(\mathcal{D}_{F,\omega^{2j}}(T)\big)=0$ for all $j\in\{1,\dots,\frac{p-1}{2}\}$;
otherwise $p$ is {\it $F$-irregular}.

\smallskip
(ii) In particular, if one defines
$\lambda_p^{\text{\rm tot}}\big(F\big) :=
\sum_{j=1}^{\frac{p-1}{2}} \;\lambda\big(\mathcal{D}_{F,\omega^{2j}}(T)\big)\in\mathbb{Z}_{\geq 0}$
then the prime $p$ is $F$-regular if and only if  both $p\nmid\text{\rm disc}_F$ and $\lambda_p^{\text{\rm tot}}(F)$ is equal to zero.
\end{dfn}

We now study totally real abelian extensions $F/\QQ$ of order $m$. Suppose that $\text{Gal}(F/\QQ)$ is a direct sum of cyclic groups of order $m_i$
where each $i\in\{1, \dots , n\}$.
Recall from earlier that Conjecture \ref{conj3} predicted for $p\gg 0$ with $p\nmid m\cdot\text{\rm disc}_F$ that
$$
\text{\rm Prob}\big(\text{\rm $p$ is $F$-regular}\big) \;\approx\; \exp\left(\!-\frac{\prod_{i=1}^n \gcd(m_i , p-1)}{2} \right)
$$
where the probability is computed over a random field extension (of fixed order), and a randomly chosen prime number.

In the sequel, we provide some theoretical evidence for the  conjecture above.
Firstly a prime $p$ satisfying $p \nmid m\cdot\text{disc}_F$ is $F$-regular if and only if
$\lambda_p(\chi \omega^{2j}) = 0$ for all characters $\chi$ of $\text{Gal}(F/\QQ)$ and $j \in \{1, \dots, \frac{p-1}{2}\}$.
If the random matrix model is accurate for a fixed prime $p$, then the probability of $\lambda_p(\chi \omega^{2j})=0$ is given by $\rho(q,0)$ where $q = p^f=\#(\mathcal{O}_{\chi}/\varpi)$.
Assuming that the distribution of the $\lambda$-invariants for characters $\chi$ of $\text{Gal}(F/\QQ)$ in the same Galois orbit is independent\footnote{This is probably true for totally split
primes in $\QQ(\chi)$, as the decomposition group is trivial. However this is not true for inert primes,
but luckily Prob$\big(\lambda_p(\chi \omega^{2j})=0\big)$ is close to $1$ if $f>1$.} we should then have an approximation
\begin{equation}\label{ProdApprox}
\text{\rm Prob}\big(\text{$p$ is $F$-regular}\big) \;\approx\! \prod_{\substack{\chi:\text{Gal}(F/\QQ)\rightarrow\overline{\mathbb{Q}}^{\times}\!\!, \\ \#(\mathcal{O}_{\chi}/\varpi) = q}} \!\rho(q,0)^{\frac{p-1}{2}}  .
\end{equation}
We already saw from Equation (\ref{EqBnds}) that (as in the proof of Theorem \ref{TotThm}):
$$
\;\beta\big(p^f\big)^{\frac{p-1}{2}}\!\cdot \prod_{t\geq 2} \big(1-p^{-tf}\big)
\;\;<\;\; \rho\big(p^f,0\big)^{\frac{p-1}{2}}
\;\;<\;\; \beta\big(p^f\big)^{\frac{p-1}{2}}\!\cdot  \prod_{t\geq 3} \big(1-p^{-tf}\big) .
$$
Furthermore, one has $\lim_{p\rightarrow\infty} \beta\big(p^f\big)^{\frac{p-1}{2}}=1$ if $f>1$ and $\lim_{p\rightarrow\infty} \beta(p)^{\frac{p-1}{2}}=e^{-1/2}$,
in which case $\lim_{p\rightarrow\infty}\rho(p,0)^{\frac{p-1}{2}}=e^{-1/2}$
while $\lim_{p\rightarrow\infty}\rho(p^f,0)^{\frac{p-1}{2}}=1$ when $f>1$.
Thus we obtain a factor of $e^{-1/2}$ for each character $\chi$ of $\text{Gal}(F/\QQ)$ such that $\#(\mathcal{O}_{\chi}/\varpi)=p$, and the other characters yield a factor of $1$
for the product in (\ref{ProdApprox}).
Consequently, it is enough to count up the characters $\chi$ such that $\mathcal{O}_{\chi}$ has residue class degree $f=1$, i.e. those $\chi$ at which $p$ is totally split inside $\QQ(\chi) = \QQ(\mu_{\ord(\chi)}) $.

\begin{lem} Let $p$ be a prime number, and let $G$ be a cyclic group of order $m$ such that $p \nmid m$. Then there exist
precisely $\gcd(m,p-1)$ distinct characters $\chi$ of $G$ having the property
that $p$ is totally split inside the cyclotomic extension $\QQ(\chi)/\QQ$. More generally, if $G$ is a direct product of cyclic groups of order $m_i$ where each $i\in\{1, \dots , n\}$, then
the number of such characters is $\prod_{i=1}^n \gcd(m_i, p-1)$.
\end{lem}

\begin{proof}
The prime $p$ totally splits in $\QQ(\mu_N) / \QQ$ if and only if $p-1 \equiv 0 \;(\!\!\!\!\mod N)$. For each character $\chi$ of a cyclic group of order $m$ clearly one has $\ord(\chi) \big| m$, and the largest
possible order of a character at which $p$ is totally split is $\gcd(m,p-1)$. Indeed such a character generates the (cyclic subgroup of) characters $\chi$ at which $p$ totally splits
in $\QQ(\chi)$.
We can decompose a character $\chi$ of a finite abelian group $G$ into a product $\chi_1  \cdots \chi_n$, where the $\chi_i$'s are characters on the cyclic components. We have $\ord(\chi) = \text{lcm}( \ord(\chi_1), \dots , \ord(\chi_n)) = N$ say: by the Chinese remainder theorem and its generalisation to non-coprime moduli,
$$ p-1  \equiv 0 \;(\!\!\!\!\!\mod N) \;\Longleftrightarrow\;  p-1 \equiv 0 \;\big(\!\!\!\!\!\mod \ord(\chi_i)\big) \;\text{ for all } i\in\{1, \dots, n\} .$$
There are $\gcd(m_i, p-1)$  distinct characters $\chi_i$ such that $p-1 \equiv 0 \;\big(\!\!\!\!\mod \ord(\chi_i)\big)$, which completes
the proof of the lemma.
\end{proof}

\begin{exm*} Suppose that $F/\QQ$ denotes a totally real cyclic extension of order $l^n$ where $l \neq p$ is a prime. If $n(p):=\min\{\text{\rm ord}_l(p-1),n\}$ then $p$ splits completely in $\mathbb{Q}(\mu_{l^{n(p)}})$ i.e.
$(p)= \mathfrak{p}_1\cdots \mathfrak{p}_{\varphi(l^{n(p)})}$, and moreover
the prime ideals $\mathfrak{p}_j$ are inert in the extension $\mathbb{Q}(\mu_{l^n})\big/\mathbb{Q}(\mu_{l^{n(p)}})$ for
each $j$. We must then have $\gcd(l^n,p-1)=l^{n(p)}$, and hence
$$
\text{\rm Prob}\big(\text{$p$ is $F$-regular}\big) \;\approx\; e^{-\frac{l^{n(p)}}{2}}  .
$$
It follows directly that the probability that $p$ is $F$-regular stabilises in any tower of extensions of $l$-power order. For instance if $l \nmid p-1$ then this probability is $e^{-1/2}$, the same as for the classical case $F=\QQ$.
\end{exm*}

\section{The general method and a summary of the data}\label{Sect3}
\noindent
We compile some evidence supporting Conjectures \ref{conj1}-\ref{conj3} by computing $\lambda_p(\chi)$ for a large number of characters $\chi$ of a given order, as well as their Teichm\"uller twists.
Let us make a few general comments. Firstly, the SageMath class {\ttfamily DirichletGroup} was used to find the distinct Dirichlet characters $\chi$ of a given conductor and order. In our
numerical experiments, we usually looked at characters $\theta$ with $\cond(\theta)<10^5$ and $\gcd\big(p,\cond(\theta)\big)=1$, and at all Teichm\"uller twists $\chi= \theta \omega^i$ so that $\chi$ was
even. To compute $\lambda_p(\chi)$ itself we employed two different methods, as we will now outline.

\smallskip
{\bf Method I - Interpolation of Bernoulli numbers.}
One commences by choosing $15$ separate values $n=i+k(p-1)$  of the power series $\mathcal{G}_{\theta \omega^i}(T) \in \mathcal{O}_{\chi}[\![T]\!]$. Applying
Equation (\ref{InterpFormula}), we immediately discover that
$$
\mathcal{G}_{\theta \omega^i}\big((1+p)^{n-1} - 1\big)  =\;   - \big( 1 - \chi(p) p^{n-1} \big) \times \frac{B_{n,\theta}}{n}
\quad\text{for positive $n \equiv i \;(\!\!\!\!\!\mod p-1)$. }
$$
An interpolation of these 15 values produces a Lagrange polynomial in $\QQ_p(\chi)\big[T\big]$,
and to calculate $\lambda_p(\chi)$ we need only look at its first few coefficients modulo $\varpi$.

As is described at length in \cite[Section 5.2]{EJV} for the ring $\mathcal{O}_{\psi} = \Zp$,
the difference between the $t$-th coefficient of the power series and the interpolated polynomial lies in $p^{C-t} \Zp$ where $C$ denotes the number of values used for the interpolation.
The statement easily extends to the valuation ring $\mathcal{O}_{\chi}$, and so we can compute $\lambda$-invariants $<C=15$ if we work at the required $p$-adic precision (see also \cite{ErMe}).

The numbers $B_{n,\theta}$ themselves are obtained by calling the SageMath function {\ttfamily bernoulli} on a {\ttfamily DirichletGroup} object.
The generalised Bernoulli numbers are elements of the global field $\QQ(\zeta_{\ord(\chi)})$, which we need to embed into a $p$-adic field. To this end, we exploit SageMath's implementation
{\ttfamily Qq} of field extensions of $\QQ_p$: we compute the extension $K=\QQ_p(\chi)$, so that $B_{n,\theta}$ is
sent to an element in $K$.

\smallskip
{\bf Method II - $p$-adic Dirichlet series expansions.}
Over the last fifteen years or so, it was found by various authors that the Kubota-Leopoldt zeta-function exhibits a nice expansion as a non-archimedean Dirichlet series \cite{De1,De2,DeQi,KnWa,Zh}.
For $x\in\mathbb{Z}_p^{\times}$ we write $\langle x\rangle=x\cdot\omega(x)^{-1}\in 1+p\mathbb{Z}_p$ for its image in the principal units.
If $\mathcal{G}_{\chi,\beta}^{\natural}(T)\in\mathcal{O}_{\chi}[\![T]\!]$ is the power series corresponding to the Iwasawa function
$$
\big(2\omega^{\beta}(2)\langle 2\rangle^{-s}-1\big)\cdot \mathbf{L}_{p}\big(s,\chi\omega^{1+\beta}\big)
$$
and assuming the prime $p\nmid \cond(\chi)$, then by \cite[Theorem 1]{DeQi} there are congruences
$$
\mathcal{G}_{\chi,\beta}^{\natural}(T) \;\equiv\;\; \sum_{j=0}^{p^N} \; c_{j}^{(N)}\!\big(\mathcal{G}_{\chi,\beta}^{\natural}\big)\cdot T^j  \mod \mathcal{J}_N
\quad\text{for all integers $N\geq 1$,}
$$
where $\mathcal{J}_N$ denotes the $\mathcal{O}_{\chi}[\![T]\!]$-ideal $\prod_{j=1}^{N}\big(T^{p^{j-1}}\!,p\big)
=\big(T,p\big)\cdot \big(T^p,p\big)\dots \big(T^{p^{N-1}},p\big)$.
These $p$-adic coefficients $c_{j}^{(N)}$ can be easily calculated via Equation (2) of {\it op. cit.} and individually require roughly
$N\cdot \log(p)\cdot p^N$ operations to work out each time.
Also, the ideals $\{\mathcal{J}_N\}$ form a decreasing set of neighborhoods of zero as $N\rightarrow\infty$.
Put $N=2$: by Proposition 1  of {\it op. cit.}, if $\text{\rm ord}_p\big(c_{j}^{(2)}\!(\mathcal{G}_{\chi,\beta}^{\natural}) \big)=0$ for some $j\leq p$ then
$$
\lambda_p(\chi\omega^{1+\beta}) \;=\; \min\Big\{j\geq 0\;\Big| \; \text{\rm ord}_p\big(c_{j}^{(2)}\!(\mathcal{G}_{\chi,\beta}^{\natural}) \big)=0\Big\}
-  \begin{cases}
1 & \text{if $\beta\equiv -1\mod p-1$} \\
0 & \text{if $\beta\not\equiv -1\mod p-1$.}
\end{cases}
$$
Provided the $\lambda$-invariant does not exceed $p$ (which is most often the case anyway),
the above formula provides a simple and efficient method of determining $\lambda_p(\chi)$ using approximately $O\big(\log(p)\cdot p^2)$ arithmetic operations.

\smallskip
Alternatively,  one may instead use a complementary expansion derived by the second author and Washington in \cite{KnWa}, namely
$$
-\big(1- \theta\omega^i (2) \langle 2 \rangle^{1-s}\big) \cdot \mathbf{L}_p(s, \theta \omega^i) \;=\;\; \sum^F_{\substack{a=1 \\ p\, \nmid\, a}}  \frac{(-1)^a}{2}\,  \theta \omega^{i-1}(a) \langle a \rangle^{-s} + \delta(F/p^2)
$$
where $s \in \Cp$ with $|s|<p^{(p-2)/(p-1)}$, $F$ is any chosen multiple of $p \cdot \cond(\theta)$, and one has $|\delta(F/p^2)| \leq |F/p^2|$. The regularisation at $c=2$ works best for odd characters $\theta$, but other choices for $c$ are also possible.
If we set $T =(1+p d)^s - 1$ with $d=\cond(\theta)$ then we have $s = \frac{\log_p(1+T)}{\log_p(1+p d)}$, and thereby define
the quantity $l(x) = - \frac{ \log_p\langle x \rangle } { \log_p(1+pd)}$ for
all $x \in \Zp^{\times}$.
The above Dirichlet series expansion yields for $T \in \Cp$ with $|T|<p^{-1/(p-1)}$:
\begin{eqnarray*}
 - \big(1 - \theta \omega^i (2) \langle 2 \rangle (1+T)^{l(2)}\big)\cdot   \mathcal{G}_{\theta \omega^i}(T) \;=\; \sum^F_{\substack{a=1 \\ p\, \nmid\, a}}  \frac{(-1)^a}{2} \, \theta \omega^{i-1}(a) (1+T)^{l(a)}  + \delta(F/p^2) \; & & \\
= \;\sum_{j=0}^{\infty} \left( \sum^F_{\substack{a=1 \\ p\, \nmid\, a}}  \frac{(-1)^a}{2}\, \theta \omega^{i-1}(a)  \binom{l(a)}{j} \right) T^j + \delta(F/p^2)  . & &
\end{eqnarray*}
If we put $F= d \cdot p^N$ with $N$ large enough, the inner bracket gives an easily computable formula for the initial coefficients of the power series $\mathcal{G}_{\theta \omega^i}(T)$(multiplied with the regularisation factor), from which we
readily obtain the value of $\lambda_p(\theta \omega^i)$.

\begin{rems*} (i) In practice we found that Method II worked perfectly well if both $p$ and $\cond(\chi)$ were not too big, but
Method I was quicker in the long term. The computations for Conjectures \ref{conj1}, \ref{conj3} and the tables of $\lambda$-invariants were mainly carried out with Method I, while Method II was used for random checks. Regarding the proportion of $\chi$-regular primes (Conjecture \ref{conj2}), it suffices to check whether the associated $\lambda_p$-invariants are zero, which can be done using the $p$-adic valuation of Bernoulli numbers. 

\smallskip
(ii)  To apply Method I for a fixed $\chi$, one simply generates a vast array of Bernoulli numbers (once and for all), and then varying the prime $p$ is not an issue.

\smallskip
(iii)  In work in progress of the authors and Luochen Zhao, using \cite{Ro,Zh} we treat certain totally real fields $F$ which are non-abelian extensions of $\QQ$, and study the
proportion of $F$-regular primes ({\it N.B.} this is {\it not} covered by Conjecture \ref{conj3}).

\smallskip
(iv) The SageMath code which implements Methods I and II is freely available from the website
\url{https://github.com/knospe/iwasawa}
\end{rems*}


\subsection{Numerical evidence for Conjecture \ref{conj1}} \label{Sect3.1}
We have computed the distribution of $\lambda$-invariants for all primitive even characters $\chi = \theta \omega^i$
satisfying $\gcd\!\big(\cond(\theta),p\big)\!=1$ and $\cond(\theta)<10000$, where we fix both $p$ and the order of $\chi$.
The tables contain the predicted and the computed proportion for $N$ characters and all relevant Teichm\"uller twists $i$
(we omit the case of a trivial zero where $\lambda \geq 1$).
 These computed numbers agree with the prediction, except for some mysterious larger deviation for the combination $p=3$ and $\ord(\chi)=2$. The proportion with $\lambda=0$ is systematically slightly underestimated.  For quadratic characters with large conductors, the computations in \cite{EJV,KrWa}  indicate convergence towards the predicted value. A possible explanation for this discrepancy might also be given by Garton's formula $\prod_{i=1}^{\infty} (1+p^{-i})^{-1}$ for the generalized symplectic case when the multiplier is congruent to $1$ mod $p$ (see \cite[Corollary 5.2.1]{Gar}); for $p=3$ this yields a probability of $0.639$.

\bigskip
\noindent $p=3$, $\ord(\chi)=2$ (even and odd quadratic characters), $f=1$

\noindent \begin{tabular}{|l|l|l|l|l|l|l|l|l|l|} \hline
$i$ & $N$ & $0$ & $1$ & $2$ & $3$ & $4$ & $5$ & $6$ & $7$ \\ \hline
  & pred. & $0.5601$ & $0.2801$ & $0.1050$ & $0.0364$ & $0.0123$ & $0.0041$ & $0.0014$ & $0.0005$ \\ \hline
$0$& $2280$ & $0.6461$ & $0.2439$ & $0.0798$ & $0.0219$ & $0.0066$ & $0.0013$ & $0.0004$ & $0.0000$  \\ \hline
$1$ & $1140$ & $0.6544$ & $0.2412$ & $0.0702$ & $0.0228$ & $0.0079$ & $0.0026$ & $0.0009$ & $0.0000$ \\ \hline
\end{tabular} \\

\medskip
\noindent $p=5$, $\ord(\chi)=2$ (even and odd quadratic characters), $f=1$

\noindent \begin{tabular}{|l|l|l|l|l|l|l|l|l|l|} \hline
$i$ & $N$ & $0$ & $1$ & $2$ & $3$ & $4$ & $5$ & $6$ & $7$ \\ \hline
  & pred. & $0.7603$ & $0.1901$ & $0.0396$ & $0.0080$ & $0.0016$ & $0.0003$ & $0.0001$ & $0.0000$ \\ \hline
$0$ &  $2535$ & $0.7854$ & $0.1732$ & $0.0316$ & $0.0087$ & $0.0012$ & $0.0000$ & $0.0000$ & $0.0000$ \\ \hline
$1$ & $1267$ & $0.8082$ & $0.1650$ & $0.0205$ & $0.0063$ & $0.0000$ & $0.0000$ & $0.0000$ & $0.0000$ \\ \hline
$2$ & $2535$ & $0.7953$ & $0.1692$ & $0.0292$ & $0.0051$ & $0.0012$ & $0.0000$ & $0.0000$ & $0.0000$   \\ \hline
$3$ & $2534$ & $0.8062$ & $0.1539$ & $0.0328$ & $0.0059$ & $0.0012$ & $0.0000$ & $0.0000$ & $0.0000$ \\ \hline
\end{tabular}\\

\medskip
\noindent $p=5$, $\ord(\chi)=3$ (even cubic characters), $f=2$

\noindent \begin{tabular}{|l|l|l|l|l|l|l|l|l|l|} \hline
$i$ & $N$ & $0$ & $1$ & $2$ & $3$ & $4$ & $5$ & $6$ & $7$ \\ \hline
& pred. & $0.9584$ & $0.0399$ & $0.0016$ & $0.0001$ & $0.0000$ & $0.0000$ & $0.0000$ & $0.0000$ \\ \hline
$0$ & $3184$ & $0.9579$ & $0.0408$ & $0.0013$ & $0.0000$ & $0.0000$ & $0.0000$ & $0.0000$ & $0.0000$ \\ \hline
$2$ & $3184$ & $0.9579$ & $0.0415$ & $0.0006$ & $0.0000$ & $0.0000$ & $0.0000$ & $0.0000$ & $0.0000$ \\ \hline

\end{tabular}\\

\medskip
\noindent $p=7$, $\ord(\chi)=3$ (even cubic characters), $f=1$

\noindent \begin{tabular}{|l|l|l|l|l|l|l|l|l|l|} \hline
$i$ &  $N$ & $0$ & $1$ & $2$ & $3$ & $4$ & $5$ & $6$ & $7$ \\ \hline
& pred. & $0.8368$ & $0.1395$ & $0.0203$ & $0.0029$ & $0.0004$ & $0.0001$ & $0.0000$ & $0.0000$ \\ \hline
$0$ & $2470$ & $0.8336$ & $0.1368$ & $0.0263$ & $0.0028$ & $0.0004$ & $0.0000$ & $0.0000$ & $0.0000$  \\ \hline
$2$ & $2470$ & $0.8490$ & $0.1283$ & $0.0190$ & $0.0032$ & $0.0004$ & $0.0000$ & $0.0000$ & $0.0000$  \\ \hline
$4$ & $2470$ & $0.8559$ & $0.1219$ & $0.0202$ & $0.0020$ & $0.0000$ & $0.0000$ & $0.0000$ & $0.0000$ \\ \hline
\end{tabular} \\

\medskip
\noindent $p=11$, $\ord(\chi)=3$ (even cubic characters), $f=2$

\noindent \begin{tabular}{|l|l|l|l|l|l|l|l|l|l|} \hline
$i$ & $N$ & $0$ & $1$ & $2$ & $3$ & $4$ & $5$ & $6$ & $7$ \\ \hline
& pred. & $0.9917$ & $0.0083$ & $0.0001$ & $0.0000$ & $0.0000$ & $0.0000$ & $0.0000$ & $0.0000$ \\ \hline
$0$ & $3184$ & $0.9900$ & $0.0101$ & $0.0000$ & $0.0000$ & $0.0000$ & $0.0000$ & $0.0000$ & $0.0000$ \\ \hline
$2$ & $3184$ & $0.9956$ & $0.0044$ & $0.0000$ & $0.0000$ & $0.0000$ & $0.0000$ & $0.0000$ & $0.0000$ \\ \hline
$4$ & $3184$ & $0.9912$ & $0.0088$ & $0.0000$ & $0.0000$ & $0.0000$ & $0.0000$ & $0.0000$ & $0.0000$ \\ \hline
$6$ & $3184$ & $0.9931$ & $0.0069$ & $0.0000$ & $0.0000$ & $0.0000$ & $0.0000$ & $0.0000$ & $0.0000$ \\ \hline
$8$ & $3184$ & $0.9962$ & $0.0038$ & $0.0000$ & $0.0000$ & $0.0000$ & $0.0000$ & $0.0000$ & $0.0000$ \\ \hline

\end{tabular}\\

\newpage
\noindent $p=13$, $\ord(\chi)=3$ (even cubic characters), $f=1$

\noindent \begin{tabular}{|l|l|l|l|l|l|l|l|l|l|} \hline
$i$ &  $N$ & $0$ & $1$ & $2$ & $3$ & $4$ & $5$ & $6$ & $7$ \\ \hline
 & pred. & $0.9172$ & $0.0764$ & $0.0059$ & $0.0005$ & $0.0000$ & $0.0000$ & $0.0000$ & $0.0000$ \\ \hline
$0$ & $2742$ & $0.9168$ & $0.0762$ & $0.0069$ & $0.0000$ & $0.0000$ & $0.0000$ & $0.0000$ & $0.0000$ \\ \hline
$2$ & $2742$ & $0.9172$ & $0.0766$ & $0.0055$ & $0.0007$ & $0.0000$ & $0.0000$ & $0.0000$ & $0.0000$ \\ \hline
$4$ & $2742$ & $0.9209$ & $0.0715$ & $0.0069$ & $0.0007$ & $0.0000$ & $0.0000$ & $0.0000$ & $0.0000$ \\ \hline
$6$ & $2742$ & $0.9263$ & $0.0686$ & $0.0047$ & $0.0004$ & $0.0000$ & $0.0000$ & $0.0000$ & $0.0000$ \\ \hline
$8$ & $2742$ & $0.9238$ & $0.0697$ & $0.0058$ & $0.0007$ & $0.0000$ & $0.0000$ & $0.0000$ & $0.0000$ \\ \hline
$10$ & $2742$ & $0.9154$ & $0.0806$ & $0.0029$ & $0.0011$ & $0.0000$ & $0.0000$ & $0.0000$ & $0.0000$  \\ \hline
\end{tabular}\\

\smallskip
\noindent $p=5$, $\ord(\chi)=4$ (even and odd quartic characters), $f=1$

\noindent \begin{tabular}{|l|l|l|l|l|l|l|l|l|l|} \hline
$i$ & $N$ & $0$ & $1$ & $2$ & $3$ & $4$ & $5$ & $6$ & $7$ \\ \hline
  & pred. & $0.7603$ & $0.1901$ & $0.0396$ & $0.0080$ & $0.0016$ & $0.0003$ & $0.0001$ & $0.0000$ \\ \hline
$0$ & $5038$ &  $0.7848$ & $0.1757$ & $0.0308$ & $0.0073$ & $0.0014$ & $0.0000$ & $0.0000$ & $0.0000$ \\ \hline
$1$ & $4018$ & $0.7867$ & $0.1737$ & $0.0299$ & $0.0075$ & $0.0017$ & $0.0005$ & $0.0000$ & $0.0000$  \\ \hline
$2$ & $5038$ & $0.7823$ & $0.1739$ & $0.0357$ & $0.0058$ & $0.0022$ & $0.0000$ & $0.0000$ & $0.0002$  \\ \hline
$3$ & $5056$ & $0.7777$ & $0.1719$ & $0.0396$ & $0.0089$ & $0.0020$ & $0.0000$ & $0.0000$ & $0.0000$  \\ \hline
\end{tabular}\\

\smallskip
\noindent $p=11$, $\ord(\chi)=5$ (even quintic characters), $f=1$

\noindent \begin{tabular}{|l|l|l|l|l|l|l|l|l|l|} \hline
$i$ & $N$ & $0$ & $1$ & $2$ & $3$ & $4$ & $5$ & $6$ & $7$ \\ \hline
& pred. & $0.9008$ & $0.0901$ & $0.0083$ & $0.0008$ & $0.0001$ & $0.0000$ & $0.0000$ & $0.0000$ \\ \hline
$0$ & $1944$ & $0.9059$ & $0.0875$ & $0.0062$ & $0.0005$ & $0.0000$ & $0.0000$ & $0.0000$ & $0.0000$  \\ \hline
$2$ & $1944$ & $0.8935$ & $0.0972$ & $0.0077$ & $0.0010$ & $0.0005$ & $0.0000$ & $0.0000$ & $0.0000$ \\ \hline
$4$ & $1944$ & $0.8997$ & $0.0911$ & $0.0082$ & $0.0005$ & $0.0005$ & $0.0000$ & $0.0000$ & $0.0000$ \\ \hline
$6$ & $1944$ & $0.9012$ & $0.0844$ & $0.0129$ & $0.0015$ & $0.0000$ & $0.0000$ & $0.0000$ & $0.0000$ \\ \hline
$8$ & $1944$ & $0.9105$ & $0.0813$ & $0.0072$ & $0.0010$ & $0.0000$ & $0.0000$ & $0.0000$ & $0.0000$ \\ \hline
\end{tabular}\\

\smallskip
\noindent $p=3$, $\ord(\chi)=8$ (even and odd octal characters), $f=2$

\noindent \begin{tabular}{|l|l|l|l|l|l|l|l|l|l|} \hline
$i$ & $N$ & $0$ & $1$ & $2$ & $3$ & $4$ & $5$ & $6$ & $7$ \\ \hline
& pred. & $0.8766$ & $0.1096$ & $0.0123$ & $0.0014$ & $0.0002$ & $0.0000$ & $0.0000$ & $0.0000$ \\ \hline
$0$ & $7264$ & $0.8945$ & $0.0950$ & $0.0080$ & $0.0017$ & $0.0006$ & $0.0000$ & $0.0003$ & $0.0000$  \\ \hline
$1$ & $6656$ & $0.8846$ & $0.1019$ & $0.0126$ & $0.0009$ & $0.0000$ & $0.0000$ & $0.0000$ & $0.0000$  \\ \hline
\end{tabular}

\medskip

\subsection{Numerical evidence for Conjecture \ref{conj2}}\label{Sect3.2}
We computed the Bernoulli numbers $B_{n,\chi}$ for all primitive $\chi$ with $1<\cond(\chi)<1000$ and fixed order $\ord(\chi)$. We checked if $p$ is $\chi$-regular for the first $100$
odd primes with $p^f \equiv 1 \;\big(\!\!\!\mod \ord(\chi)\big)$ where $f$ is fixed; the conjecture states that the proportion of regular primes is $e^{-1/2} \approx 0.6065$ if $f=1$,
while for $f > 1$ the proportion should be very close to $1$.
The Kummer congruences imply that modulo $p$,
$$
 \mathbf{L}_p(0, \chi \omega^n) = - \big(1- \chi \omega^{n-1} (p)\big) B_{1, \chi \omega^{n-1}}  \;\equiv\;
- \big(1- \chi(p) p^{n-1}\big) \frac{B_{n, \chi}}{n} = \mathbf{L}_p(1-n, \chi \omega^n)
$$
for all integers $n\geq 1$.
In order to compute $\mathbf{L}_p(0, \chi \omega^n) \mod p $ for all Teichm\"uller twists such that $\chi \omega^n$ is even, it suffices to check the untwisted Bernoulli numbers $B_{n,\chi}$
mod $p$ where $n\in\{2, \dots, p-1\}$ if $\chi$ is even, and $n\in\{1, \dots, p-2\}$ if $\chi$ is odd.

\begin{rem*}
We observe that $\big(1- \chi(p) p^{n-1} \big) \not\equiv 0 \mod \varpi$ for all integers $n \geq 2$, however for $n=1$ we always obtain a trivial zero if $\chi(p)=1$.
Hence we test regularity using Bernoulli numbers $B_{n,\chi} \mod \varpi$ for all cases (including $n=1$), which is in line with the definition of generalized (irr-)regular primes
given in  \cite{Er}.
\end{rem*}


The tables below contain the number of evaluated characters, $N$ say, and the average proportion of regular primes we found experimentally. The first line of each table indicates the expected proportion.
They also include the standard deviation, $\sigma$, as well as the minimum/maximum proportion for the corresponding data set.

\medskip
\noindent $\ord(\chi)=2$ (even and odd quadratic characters), $f=1$

\noindent \begin{tabular}{|c|c|c|c|c|}   \hline
$N$ & proportion & standard dev. & min. value & max. value \\   \hline
pred. & $  0.6065$ & - & -& -\\ \hline
$607$ & $0.6086$ & $0.0486$ & $0.48$ & $0.75$ \\ \hline
\end{tabular}\\

\medskip
\noindent $\ord(\chi)=3$ (even cubic characters), primes with $f=1$

\noindent \begin{tabular}{|c|c|c|c|c|}   \hline
$N$ & proportion & standard dev. & min. value & max. value \\   \hline
pred. & $  0.6065$ & - & - & - \\ \hline
$318$ & $0.6108$ & $0.0459$ & $0.48$ & $0.77$ \\ \hline
\end{tabular}\\

\medskip
\noindent $\ord(\chi)=3$ (even cubic characters), primes with $f=2$

\noindent \begin{tabular}{|c|c|c|c|c|}   \hline
$N$ & proportion & standard dev. & min. value & max. value \\   \hline
pred. & $  1$ & - & - & -\\ \hline
$318$ & $0.9967$ & $0.0058$ & $0.97$ & $1.00$ \\ \hline
\end{tabular}\\

\medskip
\noindent $\ord(\chi)=4$ (even and odd quartic characters), primes with $f=1$

\noindent \begin{tabular}{|c|c|c|c|c|}   \hline
$N$ & proportion & standard dev. & min. value & max. value \\   \hline
pred. & $  0.6065$ & - & - & - \\ \hline
$1242$ & $0.6065$ & $0.0485$ & $0.46$ & $0.78$ \\ \hline
\end{tabular}\\

\medskip
\noindent $\ord(\chi)=4$ (even and odd quartic characters), primes with $f=2$

\noindent \begin{tabular}{|c|c|c|c|c|}   \hline
$N$ & proportion & standard dev. & min. value & max. value \\   \hline
pred. & $  1$ & - & - & -\\ \hline
$1242$ & $0.9961$ & $0.0062$ & $0.97$ & $1.00$ \\ \hline
\end{tabular}

\medskip
\subsection{Numerical evidence for Conjecture \ref{conj3}}\label{Sect3.3}
Lastly for totally real cyclic extensions $F/\QQ$ of conductor $<10^5$, we computed the $\lambda_p(\chi)$-invariants for $\chi = \theta \omega^i$ where $\theta$ is a
character factoring through Gal$(F/\QQ)$, and $\omega^i$ is an even power of the Teichm\"uller character.  We next determined the invariant
$\lambda_p^{\text{\rm tot}}(F)=\sum_{\theta,i}\lambda_p(\theta \omega^i)$
by summing over even characters $\theta\omega^i$ factoring through the group Gal$\big(F(\mu_p)/\QQ\big)$.

For small primes $p<1000$, it is known \cite{ChGo,ErMe} that the $\lambda$-invariants of $\QQ(\mu_p)$ are zero more often than generally predicted.
In fact prime numbers $p<37$ are always regular, and so here the modified prediction for $F$-regular primes should be
$$
\text{\rm Prob}\big(\text{\rm $p$ is $F$-regular}\big) \approx\;
\exp\left(\!-\frac{\gcd(m,p-1)-1}{2} \right) \quad\!\text{if the prime $p$ is $\QQ$-regular.}
$$
For instance, if $m=[F:\QQ]=2$ we should expect the proportion $e^{-1/2} \approx 0.6065$,
and if $m=[F:\QQ]=3$ and $p \equiv 1 \;(\!\!\!\!\mod 3)$ we should instead expect $e^{-1} \approx 0.3679$. If $m=[F:\QQ]=3$ and $p \equiv 2 \;(\!\!\!\!\mod 3)$ we expect a proportion close to $1$.

\medskip
The tables shown below give the number $N$ of totally real fields $F$ that we checked, as well as the distribution of the $\lambda_p^{\text{\rm tot}}(F)$-invariant
averaged over these sample fields.

\medskip
\noindent $p=5$, $[F:\QQ]=2$, $f=1$ \\
\noindent \begin{tabular}{|l|l|l|l|l|l|l|l|l|}
\hline
$N$ & $0$ & $1$ & $2$ & $3$ & $4$ & $5$ & $6$ & $7$ \\ \hline
$2535$ & $0.6252$ & $0.2694$ & $0.0781$ & $0.0209$ & $0.0051$ & $0.0012$ & $0.0000$ & $0.0000$   \\ \hline
 \end{tabular} \\

\medskip
\noindent $p=7$, $[F:\QQ]=2$, $f=1$ \\
\noindent \begin{tabular}{|l|l|l|l|l|l|l|l|l|}
\hline
$N$ & $0$ & $1$ & $2$ & $3$ & $4$ & $5$ & $6$ & $7$ \\ \hline
$2662$ & $0.6078$ & $0.2960$ & $0.0770$ & $0.0162$ & $0.0023$ & $0.0004$ & $0.0000$ & $0.0000$ \\ \hline
 \end{tabular} \\

\medskip
\noindent $p=11$, $[F:\QQ]=2$, $f=1$ \\
\noindent \begin{tabular}{|l|l|l|l|l|l|l|l|l|}
\hline
  $N$ & $0$ & $1$ & $2$ & $3$ & $4$ & $5$ & $6$ & $7$ \\ \hline
$2786$ & $0.6027$ & $0.2893$ & $0.0876$ & $0.0183$ & $0.0022$ & $0.0000$ & $0.0000$ & $0.0000$  \\ \hline
 \end{tabular} \\

\medskip
\noindent $p=13$, $[F:\QQ]=2$, $f=1$ \\
\noindent \begin{tabular}{|l|l|l|l|l|l|l|l|l|}
\hline
  $N$ & $0$ & $1$ & $2$ & $3$ & $4$ & $5$ & $6$ & $7$ \\ \hline
$2825$ & $0.6000$ & $0.2988$ & $0.0835$ & $0.0149$ & $0.0028$ & $0.0000$ & $0.0000$ & $0.0000$  \\ \hline
 \end{tabular} \\

\medskip
\noindent $p=5$, $[F:\QQ]=3$, $f=2$ \\
\noindent \begin{tabular}{|l|l|l|l|l|l|l|l|l|}
\hline
$N$ & $0$ & $1$ & $2$ & $3$ & $4$ & $5$ & $6$ & $7$ \\ \hline
$1592$ & $0.9171$ & $0.0000$ & $0.0798$ & $0.0000$ & $0.0031$ & $0.0000$ & $0.0000$ & $0.0000$ \\ \hline
 \end{tabular} \\

\medskip
\noindent $p=7$, $[F:\QQ]=3$, $f=1$ \\
\noindent \begin{tabular}{|l|l|l|l|l|l|l|l|l|}
\hline
$N$ & $0$ & $1$ & $2$ & $3$ & $4$ & $5$ & $6$ & $7$ \\ \hline
 $1235$ & $0.3595$ & $0.3482$ & $0.1870$ & $0.0696$ & $0.0219$ & $0.0097$ & $0.0040$ & $0.0000$ \\ \hline
  \end{tabular} \\

\medskip
\noindent $p=11$, $[F:\QQ]=3$, $f=2$ \\
\noindent \begin{tabular}{|l|l|l|l|l|l|l|l|l|}
\hline
$N$ & $0$ & $1$ & $2$ & $3$ & $4$ & $5$ & $6$ & $7$ \\ \hline
$1592$ & $0.9667$ & $0.0000$ & $0.0327$ & $0.0000$ & $0.0006$ & $0.0000$ & $0.0000$ & $0.0000$ \\ \hline
 \end{tabular} \\

\medskip
\noindent $p=13$, $[F:\QQ]=3$, $f=1$ \\
\noindent \begin{tabular}{|l|l|l|l|l|l|l|l|l|}
\hline
$N$ & $0$ & $1$ & $2$ & $3$ & $4$ & $5$ & $6$ & $7$ \\ \hline
$1371$ & $0.3669$ & $0.3538$ & $0.1875$ & $0.0657$ & $0.0190$ & $0.0058$ & $0.0015$ & $0.0000$ \\ \hline
 \end{tabular} \\

\bigskip
{\it Acknowledgements.} The authors warmly thank Larry Washington, Qin Chao and Luochen Zhao for their helpful comments
and suggestions related to this paper. Furthermore, we thank the reviewers for their comments.
The numerical experiments were done using SageMath, developed by William Stein.

\bibliographystyle{amsplain}


\newpage

\appendix
\section{Tables of $\lambda$-invariants for small $\chi$ and $p$.}\label{AppxA}

\noindent
We write $\zeta_m$ to denote the $m$-th root of unity $\exp\big(\frac{2\pi i} { m}\big)$.
The tables below list for primes $p=3, 5, 7$ all primitive Dirichlet characters of the form $\chi = \theta \omega^i$ satisfying:
\begin{itemize}
\item $\cond(\theta)<1000$
\item $ p \nmid \cond(\theta)$
\item $p \nmid \ord(\chi)$
\item $f<10$
\end{itemize}
such that $\lambda_p(\chi)>0$ if there is no trivial $p$-adic zero, and for $\lambda_p(\chi)>1$ if there is.
If a character is not listed then $\lambda_p(\chi)=0$ (resp. $\lambda_p(\chi)=1$ in the trivial zero case).

\smallskip
\noindent {\em N.B.} The probability of finding a primitive character $\chi = \theta \omega^i$ with residue class degree $f\geq 10$ and non-trivial $\lambda_p(\chi)$-invariant
is small according to Conjecture \ref{conj1}.\\


\bigskip
\noindent {\bf I. \ $p=3$} \\


\csvreader[longtable={|p{8mm}|p{8cm}|p{9mm}|p{9mm}|p{6mm}|p{6mm}|},  separator=semicolon, column count=7, no head,
table head=\caption{Dirichlet characters $\chi = \theta \omega^i$ with nontrivial $\lambda_3(\chi)$ invariant.\label{tab:lambda3}}\\
    \toprule\bfseries $\mathbf \lambda_3(\chi)$ &Dirichlet character $\theta$ & twist $i$  & $\ord(\chi)$ & $f$ & triv. zero \\ \midrule\endfirsthead
    \toprule\bfseries  $\mathbf \lambda_3(\chi)$ &Dirichlet character $\theta$ & twist $i$  & $\ord(\chi)$ & $f$ & triv. zero \\ \midrule\endhead
    \bottomrule\endfoot
]{lambda3-1000.csv}{}
{$\csvcoli$ &  \csvcolii  $\ $ $\csvcoliii$ & $\csvcoliv$ &  $\csvcolv$ & $\csvcolvi$ & \csvcolvii}

\newpage
\noindent {\bf II. \ $p=5$} \\


\csvreader[longtable={|p{8mm}|p{8cm}|p{9mm}|p{9mm}|p{6mm}|p{6mm}|},  separator=semicolon, column count=7, no head,
table head=\caption{Dirichlet characters $\chi = \theta \omega^i$ with nontrivial $\lambda_5(\chi)$ invariant.\label{tab:lambda5}}\\
    \toprule\bfseries $\mathbf \lambda_5(\chi)$ &Dirichlet character $\theta$ & twist $i$  & $\ord(\chi)$ & $f$ & triv. zero \\ \midrule\endfirsthead
     \toprule\bfseries  $\mathbf \lambda_5(\chi)$ &Dirichlet character $\theta$ & twist $i$  & $\ord(\chi)$ & $f$ & triv. zero \\ \midrule\endhead
    \bottomrule\endfoot
]{lambda5-1000.csv}{}
{$\csvcoli$ &  \csvcolii  $\ $ $\csvcoliii$ & $\csvcoliv$ &  $\csvcolv$ & $\csvcolvi$ & \csvcolvii}

\newpage
\noindent {\bf III. \ $p=7$} \\


\csvreader[longtable={|p{8mm}|p{8cm}|p{9mm}|p{9mm}|p{6mm}|p{6mm}|},  separator=semicolon, column count=7, no head,
table head=\caption{Dirichlet characters $\chi = \theta \omega^i$ with nontrivial $\lambda_7(\chi)$ invariant.\label{tab:lambda7}}\\
    \toprule\bfseries $\mathbf \lambda_7(\chi)$ &Dirichlet character $\theta$ & twist $i$  & $\ord(\chi)$ & $f$ & triv. zero \\ \midrule\endfirsthead
     \toprule\bfseries  $\mathbf \lambda_7(\chi)$ &Dirichlet character $\theta$ & twist $i$  & $\ord(\chi)$ & $f$ & triv. zero \\ \midrule\endhead
    \bottomrule\endfoot
]{lambda7-1000.csv}{}
{$\csvcoli$ &  \csvcolii  $\ $ $\csvcoliii$ & $\csvcoliv$ &  $\csvcolv$ & $\csvcolvi$ & \csvcolvii}

\end{document}